\documentclass[11pt]{amsart}
\usepackage[latin1]{inputenc}
\usepackage[T1]{fontenc}
\usepackage{hyperref}
\usepackage{amsfonts}
\usepackage{amsmath}
\usepackage{epsf, subfigure, verbatim}
\usepackage{amssymb}
\usepackage{amsthm}
\usepackage{latexsym}
\usepackage{layout}
\usepackage{amsrefs}

\newtheorem{theorem}{Theorem}

\newtheorem{lemma}{Lemma}[section]
\newtheorem{corollary}{Corollary}

\newtheorem{remark}{Remark}[section]

\newcommand{\reals}{\mathbb{R}}

\newcommand{\ind}{\mathbf{1}}
\newcommand{\e}{\mathbb{E}}

\newcommand{\ee}{\mathrm{e}}

\newcommand{\p}{\mathbb{P}}

\newcommand{\qscale}{W^{(q)}}
\newcommand{\pscale}{W^{(p)}}
\newcommand{\wapq}{\mathcal W_a^{(p,q)}}
\newcommand{\wapqn}{\mathcal W_{a,n}^{(p,q)}}

\newcommand{\qscaleprime}{W^{(q) \prime}}

\newcommand{\zscale}{Z^{(q)}}

\newcommand{\zp}{Z^{(p)}}
\newcommand{\zapq}{\mathcal Z_a^{(p,q)}}
\newcommand{\zapqn}{\mathcal Z_{a,n}^{(p,q)}}

\newcommand{\drift}{\mathtt d}

\begin{document}

\title[Occupation times of L\'evy processes]{Occupation times of intervals until first passage times for spectrally negative L\'evy processes}

\author[Loeffen, Renaud and Zhou]{Ronnie L. Loeffen}
\address{School of Mathematics, University of Manchester, Oxford Road, Manchester M13 9PL, United Kingdom}
\email{ronnie.loeffen@manchester.ac.uk}

\author[]{Jean-Fran\c{c}ois Renaud}
\address{D\'epartement de math\'ematiques, Universit\'e du Qu\'ebec \`a Montr\'eal (UQAM), 201 av.\ Pr\'esident-Kennedy, Montr\'eal (Qu\'ebec) H2X 3Y7, Canada}
\email{renaud.jf@uqam.ca}

\author[]{Xiaowen Zhou}
\address{Department of Mathematics and Statistics, Concordia University, 1455 de Maisonneuve Blvd W., Montr\'{e}al (Qu\'{e}bec) H3G 1M8, Canada}
\email{xzhou@mathstat.concordia.ca}

\date{\today}


\begin{abstract}
In this paper, we identify Laplace transforms of occupation times of intervals until first passage times for spectrally negative L\'evy processes. New analytical identities for scale functions are derived and therefore the results are explicitly stated in terms of the scale functions of the process. Applications to option pricing and insurance risk models are also presented.
\end{abstract}

\maketitle

\section{Introduction and main results}

In this paper, we are interested in the joint Laplace transforms of
$$
\left( \tau_0^- , \int_0^{\tau_0^-} \ind_{(a,b)} (X_s) \mathrm{d}s \right) \quad \text{and} \quad \left( \tau_c^+ , \int_0^{\tau_c^+} \ind_{(a,b)} (X_s) \mathrm{d}s \right) ,
$$
where $X = (X_t)_{t \geq 0}$ is a spectrally negative L\'evy process, where
$$
\tau_0^- = \inf \{t > 0 \colon X_t < 0 \} \quad \text{and} \quad \tau_c^+ = \inf \{t > 0 \colon X_t > c \} ,
$$
and where $0\leq a \leq b\leq c$.  Recently, Landriault \textit{et al.} \cite{landriaultetal2011} and Kyprianou \textit{et al.} \cite{kyprianouetal2012} have studied occupation times of half lines for spectrally negative L\'evy processes, though the latter article considers a more general process, namely a refracted spectrally negative L\'evy process. The main difference between this paper and the papers  \cite{landriaultetal2011} and \cite{kyprianouetal2012} is that by using some of the techniques in \cite{kyploeffen}, we find considerably simpler expressions, which further allow us to establish a more general set of identities involving occupation times of spectrally negative L\'evy processes. Note that occupation times appear both in option pricing and in insurance risk models; we will mention two applications later on.

\bigskip

We now briefly introduce spectrally negative L\'evy processes and the associated scale functions, before stating our main results. Let $X = (X_t)_{t \geq 0}$ on the filtered probability space $(\Omega, (\mathcal F_t)_{t\geq0},\mathbb P)$ be a spectrally negative L\'evy process, that is a  process with stationary and independent increments and no positive jumps. Hereby we exclude the case that $X$ is the negative of a subordinator, i.e. we exclude the case of $X$ having decreasing paths. The law of $X$ such that $X_0 = x$ is denoted by $\p_x$ and the corresponding expectation by $\e_x$. We write $\p$ and $\e$ when $x=0$. As the L\'{e}vy process $X$ has no positive jumps, its Laplace transform exists and is given by
$$
\e \left[ \mathrm{e}^{\lambda X_t} \right] = \mathrm{e}^{t \psi(\lambda)} ,
$$
for $\lambda,t \geq 0$, where
$$
\psi(\lambda) = \gamma \lambda + \frac{1}{2} \sigma^2 \lambda^2 + \int^{\infty}_0 \left( \mathrm{e}^{-\lambda z} - 1 + \lambda z \ind_{(0,1]}(z) \right) \Pi(\mathrm{d}z) ,
$$
for $\gamma \in \reals$ and $\sigma \geq 0$, and where $\Pi$ is a $\sigma$-finite measure on $(0,\infty)$ such that
$$
\int^{\infty}_0 (1 \wedge z^2) \Pi(\mathrm{d}z) < \infty .
$$
We call the measure $\Pi$   the L\'{e}vy measure of $X$, while we refer to $(\gamma,\sigma,\Pi)$  as the L\'evy triplet of $X$. Note that for convenience we define the L\'evy measure in such a way that it is a measure on the positive half line instead of the negative half line.
Further, note that $\e \left[ X_1 \right] = \psi'(0+)$.  The process $X$ has paths of bounded variation if and only if $\sigma=0$ and $\int^{1}_0 z \Pi(\mathrm{d}z)<\infty$. In that case we denote by  $\drift:=\gamma+\int^{1}_0 z \Pi(\mathrm{d}z)$ the so-called drift of $X$.

For an arbitrary spectrally negative L\'evy process, the Laplace exponent $\psi$ is strictly convex and $\lim_{\lambda \to \infty} \psi(\lambda) = \infty$. Thus, there exists a function $\Phi \colon [0,\infty) \to [0,\infty)$ defined by $\Phi(q) = \sup \{ \lambda \geq 0 \mid \psi(\lambda) = q\}$ (its right-inverse) such that
$$
\psi ( \Phi(q) ) = q, \quad q \geq 0 .
$$
We have that $\Phi(q)=0$ if and only if $q=0$ and $\psi'(0+)\geq0$.

We now recall the definition of the $q$-scale function $W^{(q)}$. For $q \geq 0$, the $q$-scale function of the process $X$ is defined on $[0,\infty)$ as the continuous function with Laplace transform on $[0,\infty)$ given by
\begin{equation}\label{def_scale}
\int_0^{\infty} \mathrm{e}^{- \lambda y} W^{(q)} (y) \mathrm{d}y = \frac{1}{\psi(\lambda) - q} , \quad \text{for $\lambda > \Phi(q)$.}
\end{equation}
This function is unique, positive and strictly increasing for $x\geq0$ and is further continuous for $q\geq0$. We extend $W^{(q)}$ to the whole real line by setting $W^{(q)}(x)=0$ for $x<0$.  We write $W = W^{(0)}$ when $q=0$.
We will also frequently use the following function
$$
\zscale(x) = 1 + q \int_0^x{W}^{(q)}(y)\mathrm dy, \quad x\in\mathbb R.
$$

We recall some of the properties of the $q$-scale function $W^{(q)}$ and its use in fluctuation theory. Most results are taken, or can easily be derived, from \cite{kyprianou2006}. The initial values of $W^{(q)}$ is known to be
\begin{equation*}
W^{(q)}(0)=
\begin{cases}
1/\drift & \text{when $\sigma=0$ and $\int_{0}^1 z \Pi(\mathrm{d}z) < \infty$},  \\
0 & \text{otherwise},
\end{cases}
\end{equation*}
where we used the following definition: $W^{(q)}(0) = \lim_{x \downarrow 0} W^{(q)}(x)$. 
Now, for $a \in \reals$, define
$$
\tau_a^- = \inf \{t > 0 \colon X_t < a \} ,
$$
and
$$
\tau_a^+ = \inf \{t > 0 \colon X_t > a \} ,
$$
with the convention $\inf\emptyset=\infty$. It is well known that, if $a>0$ and $x \leq a$, then the solution to the two-sided exit problem is given by
\begin{equation}\label{E:exitabove}
\e_x \left[ \ee^{-q \tau_a^+} ; \tau_a^+ < \tau_0^- \right] = \frac{W^{(q)}(x)}{W^{(q)}(a)} ,
\end{equation}
\begin{equation}\label{E:exitbelow}
\e_x \left[ \ee^{-q \tau_0^-} ; \tau_0^- < \tau_a^+ \right] = Z^{(q)}(x) - Z^{(q)}(a) \frac{W^{(q)}(x)}{W^{(q)}(a)} ,
\end{equation}
where, for a random variable $Y$ and an event $A$, $\e [Y;A] := \e [Y \ind_A]$. Also, it is known that, for $a \leq x \leq b$ and $f$ a positive, measurable function, we have
\begin{multline}\label{E:discounted_deficit}
\e_x \left[ \mathrm{e}^{-q \tau_a^-}f(X_{\tau_a^-});  \tau_a^- < \tau_b^+ \right]  = f(a) \frac{\sigma^2}{2} \left[ \qscaleprime(x-a) - \qscale(x-a) \frac{\qscaleprime(b-a)}{\qscale(b-a)} \right] \\
+ \int_0^{b-a} \mathrm{d}y\int_{y}^\infty f(y-\theta+a) \Pi(\mathrm{d}\theta) \left[ \frac{ \qscale (b-a-y)}{\qscale(b-a)}\qscale(x-a) - \qscale(x-a-y) \right],
\end{multline}
where $\qscaleprime(x)$ is the derivative of $W^{(q)}(x)$, which is well-defined if $\sigma>0$. The first term of this identity corresponds to the case when $X_{\tau_a^-}=a$, a behaviour called creeping.

For more details on spectrally negative L\'{e}vy processes and fluctuation identities, the reader is referred to \cite{kyprianou2006}. Further information, examples and numerical techniques related to the computation of scale functions can be found in \cite{kuznetsovetal2011}.

\subsection{Main results}\label{sec_mainresults}

For our main results we first need to introduce three auxiliary functions. We note that
by taking Laplace transforms on both sides and using \eqref{def_scale} we can easily check that the following two equalities hold:
\begin{equation}\label{fullconvolutions}
 \begin{split}
 (q-p)\int_0^a W^{(p)}(a-y) W^{(q)}(y)\mathrm{d}y   = &   W^{(q)}(a) - W^{(p)}(a), \\
(q-p)\int_0^a W^{(p)}(a-y) Z^{(q)}(y)\mathrm{d}y = &  Z^{(q)}(a) - Z^{(p)}(a).
 \end{split}
\end{equation}
We now introduce the following two functions for $p,q\geq0$ and $x\in\mathbb R$,
\begin{equation}\label{convequiv2}
\begin{split}
 \wapq(x) :=&  W^{(p+q)}(x) - q \int_0^a W^{(p+q)}(x-y) \pscale(y) \mathrm{d}y \\
= &  W^{(p)}(x) + q \int_a^x W^{(p+q)}(x-y)W^{(p)}(y)\mathrm d y, \\
\zapq(x) :=&  Z^{(p+q)}(x) - q \int_0^a W^{(p+q)}(x-y) \zp(y) \mathrm{d}y \\
= &  Z^{(p)}(x) + q \int_a^x W^{(p+q)}(x-y)Z^{(p)}(y)\mathrm d y,
\end{split}
\end{equation}
where the second representations of $\wapq(x)$ and $\zapq(x)$ follow from \eqref{fullconvolutions}. We will use both representations throughout the text.
We further introduce, for $p\geq0$ and $q\in\mathbb R$ such that $p+q\geq0$, the function
\begin{equation*}
 \mathcal H^{(p,q)}(x) = \mathrm e^{\Phi(p)x} \left( 1 + q\int_0^{x}  \mathrm e^{-\Phi(p)y} W^{(p+q)}(y) \mathrm dy \right), \quad x\in\mathbb R.
\end{equation*}
Note that  $\mathcal H^{(p,q)}(x)=\mathrm e^{\Phi(p)x}$ for $x\leq0$ and that the Laplace transform of $\mathcal H^{(p,q)}$ on $[0,\infty)$ is explicitly given by
\begin{equation*}
 \int_0^\infty \mathrm e^{-\lambda x} \mathcal H^{(p,q)}(x) \mathrm dx = \frac{1}{\lambda-\Phi(p)}  \left( 1 + \frac {q  }  {\psi(\lambda)-p-q} \right), \quad \lambda>\Phi(p+q).
\end{equation*}
We now state our two main results.
\begin{theorem}\label{T:stopbelow_double}
For $0\leq a \leq b\leq c$, $p,q\geq 0$ and $0\leq x\leq c$,
\begin{multline*}
\e_x \left[ \mathrm{e}^{- p \tau_0^- - q \int_0^{\tau_0^-} \ind_{(a,b)} (X_s) \mathrm{d}s } ; \tau_0^- < \tau_c^+ \right] =
\zapq(x)    - q\int_b^x W^{(p)}(x-z)\zapq(z)\mathrm dz \\  - \frac{\zapq(c) - q\int_b^c W^{(p)}(c-z)\zapq(z)\mathrm dz }{ \wapq(c) - q\int_b^c W^{(p)}(c-z)\wapq(z)\mathrm dz} \left( \wapq(x) - q\int_b^x W^{(p)}(x-z)\wapq(z)\mathrm dz \right).
\end{multline*}
\end{theorem}
\begin{theorem}\label{T:stopabove}
For  $0\leq a \leq b\leq c$, $p,q\geq 0$ and $0\leq x\leq c$,
\begin{equation*}
\e_x \left[ \mathrm{e}^{- p \tau_c^+ - q \int_0^{\tau_c^+} \ind_{(a,b)} (X_s) \mathrm{d}s } ; \tau_c^+ < \tau_0^- \right] =  \frac{  \wapq(x)  - q\int_b^x W^{(p)}(x-z)\wapq(z)\mathrm dz }{ \wapq(c) - q\int_b^c W^{(p)}(c-z)\wapq(z)\mathrm dz }.
\end{equation*}

\end{theorem}
Note that the two theorems generalise \eqref{E:exitabove} and \eqref{E:exitbelow}. From these two theorems we can derive  the following list of corollaries.

\begin{corollary}\label{corol1}
\begin{itemize}
 \item[(i)]

For $0\leq a \leq b$ and $p,q,x\geq 0$,
\begin{multline*}
\e_x \left[ \mathrm{e}^{- p \tau_0^- - q \int_0^{\tau_0^-} \ind_{(a,b)} (X_s) \mathrm{d}s } ; \tau_0^- < \infty \right] =
\zapq(x)    - q\int_b^x W^{(p)}(x-z)\zapq(z)\mathrm dz \\  - \frac{ \frac{p}{\Phi(p)} + q \int_a^b \mathrm e^{-\Phi(p)y} \zapq(y)\mathrm dy }{ 1  + q \int_a^b \mathrm e^{-\Phi(p)y} \wapq(y)\mathrm dy } \left( \wapq(x) - q\int_b^x W^{(p)}(x-z)\wapq(z)\mathrm dz \right) ,
\end{multline*}
where $\lim_{p \to 0} p/\Phi(p) = \psi'(0+) \vee 0$ in the case $p=0$.
\item[(ii)] For $a,p,q,x\geq 0$,
\begin{multline*}
\e_x \left[ \mathrm{e}^{- p \tau_0^- - q \int_0^{\tau_0^-} \ind_{(a,\infty)} (X_s) \mathrm{d}s } ; \tau_0^- < \infty \right] \\ =
\zapq(x)       - \frac{ \frac{p+q}{\Phi(p+q)} - q\int_0^a \mathrm e^{-\Phi(p+q)y} Z^{(p)}(y)\mathrm dy }{ 1 - q\int_0^a \mathrm e^{-\Phi(p+q)y} W^{(p)}(y)\mathrm dy  }  \wapq(x).
\end{multline*}

\end{itemize}
\end{corollary}


\begin{corollary}\label{corol2}
\begin{itemize}
 \item[(i)]

For  $-\infty<a \leq b\leq c$, $p,q\geq 0$ and $x\leq c$,
\begin{equation*}
\e_x \left[ \mathrm{e}^{- p \tau_c^+ - q \int_0^{\tau_c^+} \ind_{(a,b)} (X_s) \mathrm{d}s } ;  \tau_c^+<\infty \right] =
\frac{   \mathcal H^{(p,q)}(x-a) - q\int_{b}^{x} W^{(p)}(x-y) \mathcal H^{(p,q)}(y-a) \mathrm dy }{  \mathcal H^{(p,q)}(c-a) - q\int_{b}^{c} W^{(p)}(c-y) \mathcal H^{(p,q)}(y-a)\mathrm dy }.
\end{equation*}


\item[(ii)] For  $b\leq c$, $p,q\geq 0$ and $x\leq c$,
\begin{equation*}
\e_x \left[ \mathrm{e}^{- p \tau_c^+ - q \int_0^{\tau_c^+} \ind_{(-\infty,b)} (X_s) \mathrm{d}s } ;  \tau_c^+<\infty \right] =
\frac{ \mathcal H^{(p+q,-q)}(x-b) }{ \mathcal H^{(p+q,-q)}(c-b)}.
\end{equation*}
\end{itemize}
\end{corollary}

\begin{corollary}\label{corol3}
\begin{itemize}
 \item[(i)]
Assume $\psi'(0+)>0$. Then for $-\infty<a\leq b$, $q\geq 0$ and $x\in\mathbb R$,
\begin{equation*}
\e_x \left[ \mathrm{e}^{  - q \int_0^{\infty} \ind_{(a,b)} (X_s) \mathrm{d}s }  \right] =
\frac{ Z^{(q)}(x-a) - q\int_{b}^{x} W(x-y) Z^{(q)}(y-a) \mathrm dy }{ 1+ \frac{q}{\psi'(0+)} \int_0^{b-a} Z^{(q)}(y)\mathrm dy }.
\end{equation*}

\item[(ii)]
Assume $\psi'(0+)>0$. Then for  $q\geq 0$ and $b,x\in\mathbb R$,
\begin{equation*}
\e_x \left[ \mathrm{e}^{  - q \int_0^{\infty} \ind_{(-\infty,b)} (X_s) \mathrm{d}s }  \right] =
\frac{\psi'(0+) \Phi(q)}{q} \mathcal H^{(q,-q)}(x-b).
\end{equation*}

 \item[(iii)]
Assume $\psi'(0+)<0$. Then for $-\infty<a\leq b$, $q\geq 0$ and $x\in\mathbb R$,
\begin{multline*}
\e_x \left[ \mathrm{e}^{  - q \int_0^{\infty} \ind_{(a,b)} (X_s) \mathrm{d}s }  \right] =
Z^{(q)}(x-a)  - q\int_{b}^{x} W^{}(x-y)Z^{(q)}(y-a) \mathrm dy \\
 -   \frac{   q \int_{0}^{b-a} \mathrm e^{-\Phi(0)y} Z^{(q)}(y) \mathrm dy }{\psi'(\Phi(0))   + q \int_{0}^{b-a} \mathrm e^{-\Phi(0)y} \mathcal H^{(0,q)}(y)\mathrm dy }   \left( \mathcal H^{(0,q)}(x-a) - q\int_{b}^{x} W^{}(x-y) \mathcal H^{(0,q)}(y-a)\mathrm dy \right).
\end{multline*}

\item[(iv)]
Assume $\psi'(0+)<0$. Then for  $q\geq 0$ and $a,x\in\mathbb R$,
\begin{equation*}
\e_x \left[ \mathrm{e}^{  - q \int_0^{\infty} \ind_{(a,\infty)} (X_s) \mathrm{d}s }  \right] =
Z^{(q)}(x-a)  - \frac{\Phi(q)-\Phi(0) }{ \Phi(q) }  \mathcal H^{(0,q)}(x-a). 
\end{equation*}
\end{itemize}

\end{corollary}

We remark that Corollary \ref{corol3}(ii) was derived earlier in \cite{landriaultetal2011}*{Corollary 1}.
Note that regarding Corollary \ref{corol3}, due to the long-term behaviour of $X$, if $\psi'(0+)\leq0$, then $\int_0^{\infty} \ind_{(-\infty,b)} (X_s) \mathrm{d}s=\infty$ a.s., if $\psi'(0+)\geq0$, then $\int_0^{\infty} \ind_{(a,\infty)} (X_s) \mathrm{d}s=\infty$ a.s. and if $\psi'(0+)=0$, then $\int_0^{\infty} \ind_{(a,b)} (X_s) \mathrm{d}s=\infty$ a.s..


We also mention the following useful identities,
\begin{equation}\label{convequiv1}
\begin{split}
 \wapq(x) - q\int_b^x W^{(p)}(x-z)\wapq(z)\mathrm dz
= &  W^{(p)}(x) + q\int_a^b W^{(p)}(x-z)\wapq(z)\mathrm dz, \\
 \zapq(x) - q\int_b^x W^{(p)}(x-z)\zapq(z)\mathrm dz
= &  Z^{(p)}(x) + q\int_a^b W^{(p)}(x-z)\zapq(z)\mathrm dz.
\end{split}
\end{equation}
These two identities can be proved easily by setting first $x=a=b$ in Theorem \ref{T:stopabove} and comparing with \eqref{E:exitabove} and then  setting  $x=a=b$ in Theorem \ref{T:stopbelow_double} and comparing with \eqref{E:exitbelow}. Similarly,
\begin{equation}\label{identH}
 \mathcal H^{(p,q)}(x-a) - q\int_{b}^{x} W^{(p)}(x-y) \mathcal H^{(p,q)}(y-a) \mathrm dy
= \mathrm e^{\Phi(p)(x-a)} + q\int_a^b W^{(p)}(x-y)\mathcal H^{(p,q)}_{a}(y-a) \mathrm dy,
\end{equation}
which can be proved easily by setting $x=a=b$ in Corollary \ref{corol2}(i) and comparing with  the identity in Corollary \ref{corol2}(i) for $q=0$. Note that \eqref{convequiv1} and \eqref{identH} lead to alternative identities for the main theorems and corollaries. Further, \eqref{convequiv1} and \eqref{identH} will also be used to prove Corollary \ref{corol1}(i) and Corollary \ref{corol3}(i) respectively.

\begin{remark}
\rm
The expressions appearing in  Theorems \ref{T:stopbelow_double} and \ref{T:stopabove} and Corollaries \ref{corol1}-\ref{corol3} are all given in terms of scale functions for which in general only the Laplace transform is known. However, there are plenty of examples of spectrally negative L\'evy processes for which an explicit formula  (though the degree of explicitness can vary case by case) exists for the scale function $W^{(q)}$, cf. \cite{kuznetsovetal2011}. For these examples one should then be able to get a more explicit expression for the functionals appearing in the aforementioned theorems and corollaries. On the other hand, there are good numerical methods for dealing with Laplace inversion  of the scale function (cf. \cite{kuznetsovetal2011}*{Section 5}) and these can be used to numerically evaluate the expressions in Theorems \ref{T:stopbelow_double} and \ref{T:stopabove} and Corollaries \ref{corol1}-\ref{corol3}. Although the Laplace transforms of $Z^{(q)}$ and $\mathcal H^{(p,q)}$ are known and thus these functions can be computed via a single Laplace inversion, this is not true in general for the functions $\wapq$  and $\zapq$ due to the appearance of incomplete convolutions. This also means that several of our indentities cannot be computed via a single Laplace inversion and more complicated numerical procedures involving Laplace inversion and computation of iterated integrals are needed.
\end{remark}

Our results improve the results from \cite{landriaultetal2011} and \cite{kyprianouetal2012} (in the no refraction case) in several ways. First, we consider occupation times of an arbitrary interval, not just intervals of the form $(-\infty,b)$. Second, we deal with  the case $p>0$. Third, we deal with a general starting point $x$; note that the expressions simplify when $x\leq b$ or $x\leq a$. Finally, our expressions are considerably simpler than the ones derived in \cite{landriaultetal2011} and \cite{kyprianouetal2012}. To illustrate this consider Corollary \ref{corol1}(i) with $p=0$, $a=0$ and $x=b$. Then
\begin{equation*}
\e_b \left[ \mathrm{e}^{- q \int_0^{\tau_0^-} \ind_{(0,b)} (X_s) \mathrm{d}s } ; \tau_0^- < \infty \right] =
Z^{(q)}(b)     - \frac{ \left( \psi'(0+)\vee 0 \right) + q \int_0^b \mathrm e^{-\Phi(0)y} Z^{(q)}(y)\mathrm dy }{ 1  + q \int_0^b \mathrm e^{-\Phi(0)y} W^{(q)}(y)\mathrm dy }  W^{(q)}(b),
\end{equation*}
which is a more compact expression and easier to evaluate than the one in Theorem 2 of \cite{landriaultetal2011} and Corollary 1(ii) of \cite{kyprianouetal2012} (in the no refraction case).

\bigskip

The rest of the paper is organized as follows. The main lemma needed for the proofs, which is based on some of the techniques used in \cite{kyploeffen},  is given in the next section. It is this lemma which allows us in the end to simplify the expressions obtained in  \cite{landriaultetal2011} and \cite{kyprianouetal2012}. Then in Sections 3-5 the proofs of the theorems and corollaries are given. The arguments used in Sections 3 and 4 (at least for the case where $X$ has paths of bounded variation) are similar to the ones  in  \cite{landriaultetal2011}.   Finally, in Section 6 we give two applications of our results.


\section{Main lemma}\label{S:techlemmas}


Recall that $X$ is a spectrally negative L\'evy process with L\'evy triplet $(\gamma,\sigma,\Pi)$.
For some particular functions $f$ associated with $X$, the right hand side of \eqref{E:discounted_deficit} can be written in a much nicer form (namely, \eqref{mgproperty} below) and this observation is the starting point of what leads in the end to the simple form, compared to the earlier works \cite{landriaultetal2011} and \cite{kyprianouetal2012},  of the  identities in the main theorems.

For a positive, measurable function $v^{(q)}(x)$, $x\in(-\infty,\infty)$, consider the following condition:
\begin{equation}\label{mgproperty}
\mathbb E_x \left[ \mathrm e^{-q\tau_a^-} v^{(q)}(X_{\tau_a^-}) \mathbf{1}_{\{\tau_a^-<\tau_b^+\}} \right]
= v^{(q)}(x)-\frac{W^{(q)}(x-a)}{W^{(q)}(b-a)} v^{(q)}(b), \quad 0\leq a\leq x\leq b.
\end{equation}
\begin{remark}
\rm Note that \eqref{mgproperty} implies via the Markov property, \eqref{E:exitabove} and the lack of upward jumps that the process
\begin{equation*}
t\mapsto \mathrm e^{-q(t\wedge\tau_a^-\wedge\tau_b^+)} v^{(q)} \left( X_{t\wedge\tau_a^-\wedge\tau_b^+} \right),
\end{equation*}
is a $\mathbb P_x$-martingale for all $x\in[a,b]$. Conversely, if the above displayed process is a $\mathbb P_x$-martingale for $x\in[a,b]$, then by taking expectations and the limit as $t\to\infty$, one can show that \eqref{mgproperty} is satisfied provided $v^{(q)}$ is sufficiently regular so that switching of the expectation and the limit is justified.
\end{remark}
For $q,a\geq0$, we define $\mathcal V^{(q)}_a$ to be the function space consisting of  functions $v^{(q)}(x)$ that satisfy \eqref{mgproperty}  for all $x$ and $b$ such that $a\leq x\leq b$.
We will now show that several type of functions lie in $\mathcal V_a^{(q)}$. Consider first the scale function $W^{(q)}(x)$. We have
for all  $0\leq a\leq x\leq b$ by the strong Markov property and \eqref{E:exitabove},
\begin{equation*}
\begin{split}
 \frac{ W^{(q)}(x) }{ W^{(q)}(b)} = &  \mathbb E_x \left[ \mathrm e^{-q\tau_b^+} \mathbf{1}_{\{\tau_b^+<\tau_0^-\}} \right] \\
= &  \mathbb E_x \left[ \mathrm e^{-q\tau_b^+} \mathbf{1}_{\{\tau_b^+<\tau_a^-\}} \right] + \mathbb E_x \left. \left[ \mathbb E_x \left[ \mathrm e^{-q\tau_b^+} \mathbf{1}_{\{\tau_a^-<\tau_b^+<\tau_0^-\}} \right| \mathcal F_{\tau_a^-} \right] \right] \\
= &  \frac{W^{(q)}(x-a)}{W^{(q)}(b-a)} + \mathbb E_x \left[ \mathrm e^{-q\tau_a^-} \mathbb E_{X_{\tau_a^-}} \left[ \mathrm e^{-q\tau_b^+} \mathbf{1}_{\{\tau_b^+<\tau_0^-\}}  \right] \mathbf{1}_{\{\tau_a^-<\tau_b^+\}} \right] \\
= &  \frac{W^{(q)}(x-a)}{W^{(q)}(b-a)} + \mathbb E_x \left[ \mathrm e^{-q\tau_a^-} \frac{ W^{(q)}(X_{\tau_a^-}) }{ W^{(q)}(b) } \mathbf{1}_{\{\tau_a^-<\tau_b^+\}} \right],
\end{split}
\end{equation*}
from which it follows that $W^{(q)}$ satisfies \eqref{mgproperty} and thus $W^{(q)} \in \mathcal V^{(q)}_a$ for all $q,a\geq0$. By spatial homogeneity it then follows that $x\mapsto W^{(q)}(x-y)$ lies in $\mathcal V^{(q)}_a$ for all $q\geq0$ and $0\leq y\leq a$. Let now
\begin{equation}\label{formvq}
 v^{(q)}(x)=\mathbb E_x \left[ \mathrm e^{-q\tau_0^-} f(X_{\tau_0^-}) \mathbf{1}_{\{\tau_0^-<\infty\}} \right], \quad x\in\mathbb R,
\end{equation}
for some    measurable function $f$ such that $|v^{(q)}(x)|<\infty$. Note that  $v^{(q)}(x)=f(x)$ for $x<0$. Then we have for $0\leq a\leq x$ by using the strong Markov property,
\begin{equation}\label{hulp1}
 \begin{split}
  v^{(q)}(x) = \mathbb E_x \left[ \left. \mathbb E_x \left[ \mathrm e^{-q\tau_0^-} f(X_{\tau_0^-})\mathbf{1}_{\{\tau_0^-<\infty\}} \right| \mathcal F_{\tau_a^-} \right] \right]
=  \mathbb E_x \left[ \mathrm e^{-q\tau_a^-} v^{(q)}(X_{\tau_a^-}) \mathbf{1}_{\{\tau_a^-<\infty\}} \right]
 \end{split}
\end{equation}
and therefore again by the strong Markov property and \eqref{E:exitabove}, we have for all $0\leq a\leq x\leq b$,
\begin{equation*}
 \begin{split}
  v^{(q)}(x) = &  \mathbb E_x \left[ \mathrm e^{-q\tau_a^-} v^{(q)}(X_{\tau_a^-}) \mathbf{1}_{\{\tau_a^-<\infty\}} \right] \\
= &  \mathbb E_x \left[ \mathrm e^{-q\tau_a^-} v^{(q)}(X_{\tau_a^-}) \mathbf{1}_{\{\tau_a^-<\tau_b^+\}} \right] + \mathbb E_x \left[ \left. \mathbb E_x \left[ \mathrm e^{-q\tau_a^-} v^{(q)}(X_{\tau_a^-}) \mathbf{1}_{\{\tau_b^+ < \tau_a^-<\infty\}} \right| \mathcal F_{\tau_b^+} \right] \right] \\
= &  \mathbb E_x \left[ \mathrm e^{-q\tau_a^-} v^{(q)}(X_{\tau_a^-}) \mathbf{1}_{\{\tau_a^-<\tau_b^+\}} \right] +  \mathbb E_x \left[ \mathrm e^{-q\tau_b^+} \mathbf{1}_{\{\tau_b^+<\tau_a^-\}} \right]   \mathbb E_b \left[ \mathrm e^{-q\tau_a^-} v^{(q)}(X_{\tau_a^-}) \mathbf{1}_{\{\tau_a^-<\infty\}} \right] \\
= &  \mathbb E_x \left[ \mathrm e^{-q\tau_a^-} v^{(q)}(X_{\tau_a^-}) \mathbf{1}_{\{\tau_a^-<\tau_b^+\}} \right] +  \frac{W^{(q)}(x-a)}{W^{(q)}(b-a)}   \mathbb E_b \left[ \mathrm e^{-q\tau_a^-} v^{(q)}(X_{\tau_a^-}) \mathbf{1}_{\{\tau_a^-<\infty\}} \right].
 \end{split}
\end{equation*}
Now using \eqref{hulp1} for $x=b$ for the last term on the right hand side of the previous computation, we see that $v^{(q)}(\cdot)$ of the form \eqref{formvq} satisfies \eqref{mgproperty}. In particular, for $f\equiv1$, $v^{(q)}(\cdot)$ of the form  \eqref{formvq}  lies in $\mathcal V^{(q)}_a$ for all $q,a\geq0$.
As $\mathcal V^{(q)}_a$ is a linear space it follows via \eqref{E:exitbelow} that we also have for all $q,a\geq0$,
\begin{equation*}
 Z^{(q)}(x)= \frac q{\Phi(q)} W^{(q)}(x) + \mathbb E_x \left[ \mathrm e^{-q\tau_0^-} \mathbf{1}_{\{\tau_0^-<\infty\}} \right] \in \mathcal V^{(q)}_a.
\end{equation*}

\bigskip


The proofs of the theorems and corollaries in Section \ref{sec_mainresults} and the  next lemma in the case where the process $X$ has paths of unbounded variation use an approximation argument for which we need to introduce a sequence $(X^n)_{n \geq 1}$ of spectrally negative L\'evy processes of bounded variation.
To this end, suppose $X$ is a spectrally negative L\'evy process having paths of unbounded variation. 
with L\'evy triplet $(\gamma,\sigma,\Pi)$. Form for $n\geq1$ the  spectrally negative L\'evy process $X^n=(X^n_t)_{t\geq0}$ with L\'evy triplet $(\gamma,0,\Pi_n)$, whereby $$\Pi_n(\mathrm d\theta):=\mathbf 1_{\{\theta\geq 1/n\}}\Pi(\mathrm{d}\theta) + \sigma^2 n^2\delta_{1/n}(\mathrm d\theta),$$
with $\delta_{1/n}(\mathrm d\theta)$ standing for the Dirac point mass at $1/n$. 
Note that $X^n$ has paths of bounded variation with the so-called drift  given by $\drift_n:=\gamma+\int_{1/n}^1 \theta\Pi(\mathrm d\theta) + \sigma^2 n^2$, which means that $\drift_n$ may be negative for small $n$. Though we do have that $X^n$ is a true spectrally negative L\'evy process for large enough $n$  which is all that we need. By Bertoin \cite{bertoin1996}*{p.210}, $X^n$ converges almost surely to $X$ uniformly on compact time intervals. Denote by $\mathcal V^{(q)}_{a,n}$ the function space $\mathcal V^{(q)}_a$  corresponding to $X^n$.
The following lemma is the main result of this section.
\begin{lemma}\label{mainlemma}
Let $q,a\geq0$  and $v^{(q)}$ be a positive, measurable function on $\mathbb R$. Given a spectrally negative L\'evy process $X$, consider the following assumptions:
\begin{itemize}
 \item[(i)] If $X$ has  paths of bounded variation, assume that $v^{(q)}\in\mathcal V^{(q)}_a$ and
\begin{equation}\label{existencelaplace}
 \int_0^\infty \mathrm e^{-\lambda z}v^{(q)}(z)\mathrm dz<\infty , \quad \text{for $\lambda$ large enough.}
\end{equation}
\item[(ii)] If $X$  has paths of unbounded variation,  assume that $v^{(q)}$ is continuous and that there exists a sequence of functions $v_n^{(q)}\in \mathcal V^{(q)}_{a,n}$ satisfying \eqref{existencelaplace} such that $v_n^{(q)}$ converges to $v^{(q)}$ uniformly on compact subsets, i.e.,
\begin{equation}\label{uniformconv_v}
 \lim_{n\to\infty} \sup_{x\in[x_0,x_1]} |v^{(q)}_n(x) - v^{(q)}(x)|=0,  \quad \text{for all $x_0<x_1$},
\end{equation}
and such that for all $x_0\geq0$ there exists $K_{x_0}>0$, $n_0\geq1$ such that
\begin{equation}\label{uniformbound_v}
\text{$|v^{(q)}_n(x)|\leq K_{x_0}$ \quad for all $n\geq n_0$, $x\leq x_0$}.
\end{equation}
\end{itemize}
If (i) or (ii) holds, then we have for all $p\geq0$ and $x,b$ such that $a\leq x\leq b$,
\begin{equation}\label{identmainlemma}
 \begin{split}
 \mathbb E_x \left[  \mathrm e^{-p\tau_a^-} v^{(q)}(X_{\tau_a^-})\mathbf{1}_{\{\tau_a^-<\tau_b^+\}}   \right]
 = &   v^{(q)}(x) - (q-p)\int_a^x W^{(p)}(x-y) v^{(q)}(y)\mathrm{d}y \\
\\ &  - \frac{  W^{(p)}(x-a)  }{ W^{(p)}(b-a) } \left( v^{(q)}(b) - (q-p)\int_a^b W^{(p)}(b-y) v^{(q)}(y)\mathrm{d}y \right).
\end{split}
\end{equation}
\end{lemma}
\begin{proof}[\textbf{Proof}]
We first prove (in three steps) the lemma for the case that $X$ has paths of bounded variation, i.e. $\sigma=0$ and $\int_0^1\theta\Pi(\mathrm{d}\theta)<\infty$. Recall that  $\drift=\gamma+\int_0^1\theta\Pi(\mathrm{d}\theta)>0$ is the drift of $X$.
\paragraph{Step 1} We have by \eqref{mgproperty} and  \eqref{E:discounted_deficit}, for $a \leq x\leq b$,
 \begin{equation}\label{sigma_BV1}
 \begin{split}
 v^{(q)}(x)-\frac{W^{(q)}(x-a)}{W^{(q)}(b-a)} v^{(q)}(b) = &  \mathbb E_x \left[ \mathrm e^{-q\tau_a^-} v^{(q)}(X_{\tau_a^-}) \mathbf{1}_{\{\tau_a^-<\tau_b^+\}} \right] \\
= & \int_0^\infty \int_{(y,\infty)} v^{(q)}(y-\theta+a)\Pi(\mathrm d\theta) \\
& \times  \left[ \frac{ W^{(q)}(b-a-y) }{ W^{(q)}(b-a) }  W^{(q)}(x-a) - W^{(q)}(x-a-y) \right] \mathrm dy \\
= & \int_0^\infty \int_{(y,\infty)} v^{(q)}(y-\theta+a)\Pi(\mathrm d\theta) \frac{ W^{(q)}(b-a-y) }{ W^{(q)}(b-a) }  W^{(q)}(x-a) \mathrm{d}y \\
& - \int_0^\infty \int_{(y,\infty)} v^{(q)}(y-\theta+a)\Pi(\mathrm d\theta)  W^{(q)}(x-a-y)   \mathrm dy,
 \end{split}
\end{equation}
whereby the splitting of the integral in the last line is possible due to $\int_0^1\theta\Pi(\mathrm{d}\theta)<\infty$. 
By putting $x=a$ in \eqref{sigma_BV1} and recalling $W^{(q)}(0)=1/\drift$, we get for all $b\geq a$,
\begin{equation}
\label{sigmais0_step}
  \int_0^{\infty}\int_{(y,\infty)} v^{(q)}(y-\theta+a)\Pi(\mathrm{d}\theta)W^{(q)}(b-a-y)\mathrm{d}y \\ =
\drift  W^{(q)}(b-a) v^{(q)}(a) - v^{(q)}(b).
\end{equation}

\paragraph{Step 2}
Let $\lambda_0>0$ be large enough such that the Laplace transform of $v^{(q)}(x)$ exists for $\lambda>\lambda_0$, cf. condition \eqref{existencelaplace}. Taking Laplace transforms in $b$  on both sides of \eqref{sigmais0_step} and using \eqref{def_scale} leads to, for $\lambda>\Phi(q)\vee \lambda_0$,
\begin{equation}
\label{step2_a}
\int_0^\infty \mathrm{e}^{-\lambda y} \int_{(y,\infty)} v^{(q)}(y-\theta+a)\Pi(\mathrm{d}\theta)\mathrm{d}y \\
= \drift v^{(q)}(a)-(\psi(\lambda)-q)\mathrm{e}^{\lambda a}   \int_a^\infty\mathrm{e}^{-\lambda b} v^{(q)}(b)\mathrm{d}b. 
\end{equation}
Let $p\geq0$. Then using \eqref{step2_a}, we get for $\lambda>\Phi(q)\vee \Phi(p)\vee\lambda_0$,
\begin{equation*}
\begin{split}
 \int_a^\infty \mathrm{e}^{-\lambda b} \int_0^{\infty}\int_{(y,\infty)} & v^{(q)}(y-\theta+a)\Pi(\mathrm{d}\theta)W^{(p)}(b-a-y)\mathrm{d}y\mathrm{d}b \\
= &  \frac{\mathrm{e}^{-\lambda a}}{\psi(\lambda)-p} \left( \drift v^{(q)}(a)-(\psi(\lambda)-q)\mathrm{e}^{\lambda a}  \int_a^\infty\mathrm{e}^{-\lambda b} v^{(q)}(b)\mathrm{d}b \right)\\
= &  \frac{\mathrm{e}^{-\lambda a}}{\psi(\lambda)-p}  \drift v^{(q)}(a) - \int_a^\infty\mathrm{e}^{-\lambda b} v^{(q)}(b)\mathrm{d}b +  \frac{q-p}{\psi(\lambda)-p}\int_a^\infty\mathrm{e}^{-\lambda b} v^{(q)}(b)\mathrm{d}b . 
\end{split}
\end{equation*}
Now by Laplace inversion, we get for all $b\geq a$,
\begin{multline}
\label{step2_b}
  \int_0^{\infty}\int_{(y,\infty)} v^{(q)}(y-\theta+a)\Pi(\mathrm{d}\theta)W^{(p)}(b-a-y)\mathrm{d}y  \\
= \drift v^{(q)}(a) W^{(p)}(b-a)- v^{(q)}(b)   + (q-p)\int_a^b W^{(p)}(b-y)  v^{(q)}(y) \mathrm{d}y. 
\end{multline}

\paragraph{Step 3}
We know by \eqref{E:discounted_deficit} that  for $a\leq x\leq b$,
 \begin{equation*}
\begin{split}
  \mathbb E_x & \left[  \mathrm e^{-p\tau_a^-} v^{(q)}(X_{\tau_a^-})\mathbf{1}_{\{\tau_a^-<\tau_b^+\}}   \right]  \\
= & \int_0^\infty \int_{(y,\infty)} v^{(q)}(y-\theta+a)\Pi(\mathrm d\theta) \left[ \frac{ W^{(p)}(b-a-y) }{ W^{(p)}(b-a) }  W^{(p)}(x-a) - W^{(p)}(x-a-y) \right] \mathrm dy. \\
\end{split}
 \end{equation*}
Hence using \eqref{step2_b} twice, we get the required identity
\begin{equation*}
 \begin{split}
 \mathbb E_x \left[  \mathrm e^{-p\tau_a^-} v^{(q)}(X_{\tau_a^-})\mathbf{1}_{\{\tau_a^-<\tau_b^+\}}   \right]
= &    v^{(q)}(x) - (q-p)\int_a^x W^{(p)}(x-y) v^{(q)}(y)\mathrm{d}y \\
& - \frac{  W^{(p)}(x-a)  }{ W^{(p)}(b-a) } \left( v^{(q)}(b) - (q-p)\int_a^b W^{(p)}(b-y) v^{(q)}(y)\mathrm{d}y \right).
\end{split}
\end{equation*}

We now  prove the lemma for the case that $X$ has paths of unbounded variation. Hereby we assume without loss of generality that $p>0$ as the case $p=0$ can be dealt with by taking limits as $p\downarrow0$ using the fact that $W^{(p)}(x)$ is continuous and increasing (cf. \eqref{fullconvolutions}) in $p\geq0$. We denote by $W^{(p)}_n$ the $p$-scale function corresponding to the spectrally negative L\'evy process $X^n$. 
Further, let
\begin{equation*}
 \tau_{a,n}^-=\inf\{t>0:X_t^n<a\}, \quad \tau_{b,n}^+=\inf\{t>0:X^n_t>b\}.
\end{equation*}
Then since we have proved the lemma for the case of bounded variation,
\begin{multline}\label{mainident_n}
 \mathbb E_x \left[  \mathrm e^{-p\tau_{a,n}^-} v_n^{(q)}(X^n_{\tau_{a,n}^-})\mathbf{1}_{\{\tau_{a,n}^-<\tau_{b,n}^+\}}   \right]
=     v_n^{(q)}(x) - (q-p)\int_a^x W^{(p)}_n(x-y) v^{(q)}_n(y)\mathrm{d}y \\
  - \frac{  W_n^{(p)}(x-a)  }{ W_n^{(p)}(b-a) } \left( v_n^{(q)}(b) - (q-p)\int_a^b W_n^{(p)}(b-y) v_n^{(q)}(y)\mathrm{d}y \right).
\end{multline}
We aim to prove \eqref{identmainlemma} by taking limits  as $n\to\infty$ on both sides of \eqref{mainident_n}. By p.210 of Bertoin \cite{bertoin1996}, $X^n$ converges almost surely to $X$ uniformly on compact time intervals, i.e. for all $t>0$, $\lim_{n\to\infty} \sup_{s\in[0,t]} |X^n_s- X_s|=0$, $\mathbb P_x$-a.s. This implies that for any $t>0$, $\mathbb P_x$-a.s.,
\begin{equation*}
\tau_{a,n}^-\wedge t \rightarrow \tau_a^-\wedge t, \quad \tau_{b,n}^+\wedge t \rightarrow \tau_b^+\wedge t, \quad X^n_{\tau_{a,n}^-\wedge t} \to X_{\tau_{a}^-\wedge t}
\end{equation*}
and thus by \eqref{uniformconv_v}  we  have for any $t>0$, $\mathbb P_x$-a.s.,
\begin{equation*}
 \mathrm e^{-p(\tau_{a,n}^-\wedge t)} v^{(q)}_n(X^n_{\tau_{a,n}^-\wedge t} ) \mathbf 1_{\{\tau_{a,n}^-\wedge t  < \tau_{b,n}^+\wedge t \}} \to  \mathrm e^{-p(\tau_{a}^-\wedge t)}v^{(q)}(X_{\tau_{a}^-\wedge t})  \mathbf 1_{\{\tau_{a}^-\wedge t  < \tau_{b}^+\wedge t \}}.
\end{equation*}
or equivalently
\begin{equation}\label{xiaowen1}
 \mathrm e^{-p\tau_{a,n}^-} v^{(q)}_n(X^n_{\tau_{a,n}^-} ) \mathbf 1_{\{\tau_{a,n}^-  < \tau_{b,n}^+\wedge t \}} \to  \mathrm e^{-p\tau_{a}^-}v^{(q)}(X_{\tau_{a}^-})  \mathbf 1_{\{\tau_{a}^-  < \tau_{b}^+\wedge t \}}.
\end{equation}
Notice that $\mathbb P_x$-a.s.,
\begin{equation}\label{xiaowen3}
X^n_{\tau_{a,n}^-}\mathbf 1_{\{\tau_{a,n}^-  <\tau_{b,n}^+ \}}\leq a, \quad  X_{\tau_{a}^-}\mathbf 1_{\{\tau_{a}^-  <\tau_{b}^+ \}}\leq a,
\end{equation}
which implies further in combination with the triangle inequality,
\begin{equation}\label{xiaowen2}
\begin{split}
| \mathrm e^{-p\tau_{a,n}^- } v^{(q)}_n & (X^n_{\tau_{a,n}^- } )  \mathbf 1_{\{\tau_{a,n}^-   < \tau_{b,n}^+  \}}   - \mathrm e^{-p\tau_{a}^- }v^{(q)}(X_{\tau_{a}^- })  \mathbf 1_{\{\tau_{a}^-   < \tau_{b}^+  \}} | \\
 \leq & | \mathrm e^{-p\tau_{a,n}^- } v^{(q)}_n   (X^n_{\tau_{a,n}^- } )  \mathbf 1_{\{\tau_{a,n}^-< \tau_{b,n}^+  \}}    - \mathrm e^{-p \tau_{a,n}^- } v^{(q)}_n(X^n_{\tau_{a,n}^-} ) \mathbf 1_{\{\tau_{a,n}^-  < \tau_{b,n}^+\wedge t \}} | \\
& + | \mathrm e^{-p \tau_{a,n}^- } v^{(q)}_n(X^n_{\tau_{a,n}^-} ) \mathbf 1_{\{\tau_{a,n}^-  < \tau_{b,n}^+\wedge t \}}    - \mathrm e^{-p \tau_{a}^- }v^{(q)}(X_{\tau_{a}^-})  \mathbf 1_{\{\tau_{a}^-  < \tau_{b}^+\wedge t \}} | \\
& + | \mathrm e^{-p \tau_{a}^- }v^{(q)}(X_{\tau_{a}^-})  \mathbf 1_{\{\tau_{a}^-  < \tau_{b}^+\wedge t \}}    - \mathrm e^{-p\tau_{a}^- }v^{(q)}(X_{\tau_{a}^- })  \mathbf 1_{\{\tau_{a}^-   < \tau_{b}^+  \}} | \\
= &  |\mathrm e^{-p\tau_{a,n}^- } v^{(q)}_n   (X^n_{\tau_{a,n}^- } )\mathbf 1_{\{\tau_{a,n}^-   < \tau_{b,n}^+  \}}\mathbf 1_{\{t\leq \tau_{a,n}^-\}} | \\
& + | \mathrm e^{-p \tau_{a,n}^- } v^{(q)}_n(X^n_{\tau_{a,n}^-} ) \mathbf 1_{\{\tau_{a,n}^-  < \tau_{b,n}^+\wedge t \}}    - \mathrm e^{-p \tau_{a}^- }v^{(q)}(X_{\tau_{a}^-})  \mathbf 1_{\{\tau_{a}^-  < \tau_{b}^+\wedge t \}} | \\
& + |\mathrm e^{-p\tau_{a}^- } v^{(q)}   (X_{\tau_{a}^- } )\mathbf 1_{\{\tau_{a}^-   < \tau_{b}^+  \}}\mathbf 1_{\{t\leq \tau_{a}^-\}} |  \\
 \leq & \mathrm e^{-pt} \left( \mathbf 1_{\{t\leq \tau_{a,n}^-\}} \sup_{y\leq a} |v_n^{(q)}(y)| + \mathbf 1_{\{t\leq \tau_{a}^-\}} \sup_{y\leq a} |v^{(q)}(y)|  \right) \\
& + | \mathrm e^{-p \tau_{a,n}^- } v^{(q)}_n(X^n_{\tau_{a,n}^-} ) \mathbf 1_{\{\tau_{a,n}^-  < \tau_{b,n}^+\wedge t \}}    - \mathrm e^{-p \tau_{a}^- }v^{(q)}(X_{\tau_{a}^-})  \mathbf 1_{\{\tau_{a}^-  < \tau_{b}^+\wedge t \}} |.
\end{split}
\end{equation}
By  \eqref{uniformbound_v} and \eqref{xiaowen1} we can (since we assumed $p>0$) first choose a $t$ large enough and then choose $n$ large to make the right hand side of \eqref{xiaowen2} arbitrarily small, which means that $\mathbb P_x$-a.s.,
\begin{equation*}
 \mathrm e^{-p\tau_{a,n}^- } v^{(q)}_n  (X^n_{\tau_{a,n}^- } )  \mathbf 1_{\{\tau_{a,n}^-   < \tau_{b,n}^+  \}}   \to \mathrm e^{-p\tau_{a}^- }v^{(q)}(X_{\tau_{a}^- })  \mathbf 1_{\{\tau_{a}^-   < \tau_{b}^+  \}}.
\end{equation*}
By \eqref{xiaowen3} and \eqref{uniformbound_v} in combination with the dominated convergence theorem we can then conclude that
\begin{equation*}
 \begin{split}
\lim_{n\to\infty} \mathbb E_x   \left[  \mathrm e^{-p\tau_{a,n}^-} v_n^{(q)}(X^n_{\tau_{a,n}^-})\mathbf{1}_{\{\tau_{a,n}^-<\tau_{b,n}^+\}}   \right]
=   \mathbb E_x  \left[  \mathrm e^{-p\tau_{a}^-} v^{(q)}(X_{\tau_{a}^-})\mathbf{1}_{\{\tau_{a}^-<\tau_{b}^+\}}   \right].
\end{split}
\end{equation*}

\medskip

It remains to show that the right hand side of \eqref{mainident_n} converges to the right hand side of  \eqref{identmainlemma}. It is an easy exercise to show that the Laplace exponent of $X^n$ converges to the Laplace exponent of $X$ which means via \eqref{def_scale} that the Laplace transform of $W^{(p)}_n$ converges to the Laplace transform of $W^{(p)}$. Hence by the continuity theorem of Laplace transforms (cf. \cite{fellervol2}*{Theorem 2a in Section XIII.1}), $W^{(p)}_n(x)\to W^{(p)}(x)$ for all $x\geq0$ and $p\geq0$.  Using the dominated convergence theorem in combination with \eqref{uniformconv_v}, \eqref{uniformbound_v} and the fact  that scale functions are increasing, we deduce that indeed  the right hand side of \eqref{mainident_n} converges to the right hand side of  \eqref{identmainlemma}.
\end{proof}

\begin{lemma}
The conclusion of Lemma \ref{mainlemma} holds for (i) $v^{(q)}(x)=W^{(q)}(x)$ for any  $q,a\geq0$, (ii) $v^{(q)}(x)=Z^{(q)}(x)$ for any $q,a\geq0$ and (iii) $v^{(q)}(x)=W^{(q)}(x-y)$ for any $q\geq0$ and $0\leq y\leq a$.
\end{lemma}
\begin{proof}[\textbf{Proof}]
Note that case (iii) will follow from case (i) by spatial homogeneity of a L\'evy process. For cases (i) and (ii), from the considerations in the beginning of Section \ref{S:techlemmas} and \eqref{def_scale}, the assumptions in Lemma \ref{mainlemma} are clearly satisfied when $X$ has paths of bounded variation. When $X$ has paths of unbounded variation, we let the function $v^{(q)}_n\in\mathcal V^{(q)}_{a,n}$ in case (i), respectively case (ii), be  $W^{(q)}_n$  (the $q$-scale function corresponding to $X^n$), respectively
$Z^{(q)}_n(x):=1+q\int_0^x W_n^{(q)}(y)\mathrm dy$.   We have seen in the proof of Lemma \ref{mainlemma} that $W^{(q)}_n(x)$ converges to $W^{(q)}(x)$ and  since the $q$-scale function is increasing and positive,
it follows that \eqref{uniformbound_v} is satisfied in case (i). This implies further by the dominated convergence theorem, that $Z^{(q)}_n(x)$ converges to $Z^{(q)}(x)$ and as $Z^{(q)}_n$ is also positive and increasing, \eqref{uniformbound_v} is also satisfied in case (ii).

What remains to show is that the convergence of $W^{(q)}_n$ to $W^{(q)}$ and $Z^{(q)}_n$ to $Z^{(q)}$  is actually uniform on compact subsets. Since $x\mapsto \log W^{(q)}_n(x)$ is a concave function (cf. \cite{loeffen2}*{p.89}) and converges pointwise to $\log W^{(q)}(x)$, it follows that $\log W^{(q)}_n$ converges uniformly on compact subsets to $\log W^{(q)}$, cf. \cite{robertsvarberg}*{p.17, Theorem E}. As the exponential function is locally Lipschitz, it is then easy to show that also  $W^{(q)}_n$ converges to $W^{(q)}$ uniformly on compact subsets. It then easily follows that also $Z^{(q)}_n$ converges to $Z^{(q)}$ uniformly on compact subsets.
\end{proof}

\begin{remark}
\rm The proof of Lemma \ref{mainlemma} in the bounded variation case uses very similar steps as the proof of Theorem 16 in \cite{kyploeffen}. In order to make the connection clear between these two results, let us reformulate the left hand side of \eqref{identmainlemma} in a different setting. Let $Y$ be a spectrally negative L\'evy process with  L\'evy triplet $(\gamma,\sigma,\Pi)$ and killing rate $p\geq0$, which means that $Y$ is   a spectrally negative L\'evy process   killed at an independent exponentially distributed amount of time   with parameter $p$.  Further, let   $Z$ be another  spectrally negative L\'evy process  with   L\'evy triplet  $(\gamma',\sigma',\Pi')$ and killing rate $q\geq0$. Define the first passage times,
\begin{equation*}
\begin{split}
\tau_a^- = \inf \{t > 0 \colon Y_t < a \} , \quad & \tau_b^+ = \inf \{t > 0 \colon Y_t > b \}, \\
\kappa_a^- = \inf \{t > 0 \colon Z_t < a \} , \quad & \kappa_b^+ = \inf \{t > 0 \colon Z_t > b \}.
\end{split}
\end{equation*}
and denote by $W_Y$  the scale function associated to $Y$, which is defined as the $p$-scale function $W^{(p)}$ corresponding to the  unkilled spectrally negative L\'evy process with L\'evy triplet $(\gamma,\sigma,\Pi)$. Similarly, define $W_Z$. Also, let
 $v$ be a positive, measurable function satisfying
\begin{equation*}
\mathbb E_x \left[ v(Z_{\kappa_a^-}) \mathbf 1_{\{\kappa_a^-<\kappa_b^+\}} \right] = v(x) - \frac{W_Z(x-a)}{W_Z(b-a)} v(b).
\end{equation*}
Then Lemma \ref{mainlemma} provides, under some additional regularity assumptions on $v$, an expression for the quantity
\begin{equation}\label{remarkrefracted}
\mathbb E_x \left[ v(Y_{\tau_a^-}) \mathbf 1_{\{\tau_a^-<\tau_b^+\}}  \right]
\end{equation}
in the case where $\gamma=\gamma'$, $\sigma=\sigma'$ and $\Pi=\Pi'$ (i.e. only the killing rates differ), whereas Kyprianou and Loeffen \cite{kyploeffen}*{Theorem 16} provide a similar-looking expression for \eqref{remarkrefracted} with $v=W_Z$ in the case where  $\sigma=\sigma'$, $\Pi=\Pi'$ and $p=q$ (i.e. only the first parameters of the L\'evy triplets differ).
\end{remark}

\section{Proof of Theorem~\ref{T:stopbelow_double}}

We first prove the theorem in the case where $X$ has paths of bounded variation.
Fix $0 \leq a < b$ and $p,q \geq 0$. For $x \leq c$, define
$$
w(x) = \e_x \left[ \mathrm{e}^{- p \tau_0^- - q \int_0^{\tau_0^-} \ind_{(a,b)} (X_s) \mathrm{d}s }  ; \tau_0^-<\tau_c^+ \right] .
$$
Using the strong Markov property of $X$, the fact that $X$ is skip-free upward and \eqref{E:exitabove} and \eqref{E:exitbelow}, we can write, for $x < a$,
\begin{align}\label{E:under}
w(x) = &  \e_x \left[ \mathrm{e}^{-p \tau_0^-} ; \tau_0^- < \tau_a^+ \right]  + w(a) \e_x \left[ \mathrm{e}^{-p \tau_a^+} ; \tau_a^+ < \tau_0^- \right] \notag \\
= &  \zp(x) + \left( \frac{w(a)-\zp(a)}{\pscale(a)} \right) \pscale(x) .
\end{align}
Similarly, for $a \leq x < b$, using \eqref{E:under}, we have
\begin{align}\label{E:inside}
w(x) = &  w(b) \e_x \left[ \mathrm{e}^{-(p+q) \tau_b^+} ; \tau_b^+ < \tau_a^- \right] + \e_x \left[ \mathrm{e}^{-(p+q) \tau_a^-} w \left( X_{\tau_a^-}\right) ; \tau_a^- < \tau_b^+ \right] \notag \\
= &  w(b) \frac{W^{(p+q)}(x-a)}{W^{(p+q)}(b-a)} + \e_x \left[ \mathrm{e}^{-(p+q) \tau_a^-} \zp \left( X_{\tau_a^-}\right) ; \tau_a^- < \tau_b^+ \right] \notag \\
& \qquad + \left( \frac{w(a)-\zp(a)}{\pscale(a)} \right) \e_x \left[ \mathrm{e}^{-(p+q) \tau_a^-} \pscale \left( X_{\tau_a^-}\right) ; \tau_a^- < \tau_b^+ \right] .
\end{align}
Since one can show by the lemmas in Section \ref{S:techlemmas} that
\begin{equation}\label{vqiswq}
\e_x \left[ \mathrm{e}^{-(p+q) \tau_a^-} \pscale \left( X_{\tau_a^-}\right) ; \tau_a^- < \tau_b^+ \right] = \wapq(x) - \frac{W^{(p+q)}(x-a)}{W^{(p+q)}(b-a)} \wapq(b)
\end{equation}
and
$$
\e_x \left[ \mathrm{e}^{-(p+q) \tau_a^-} \zp \left( X_{\tau_a^-}\right) ; \tau_a^- < \tau_b^+ \right] = \zapq(x) - \frac{W^{(p+q)}(x-a)}{W^{(p+q)}(b-a)} \zapq(b) ,
$$
we get, for $a \leq x < b$,
\begin{multline}\label{E:inside2}
w(x) = \frac{W^{(p+q)}(x-a)}{W^{(p+q)}(b-a)} \left( w(b) - \zapq(b)  - \left( \frac{w(a)-\zp(a)}{\pscale(a)} \right) \wapq(b) \right) \\
+ \zapq(x)   + \left( \frac{w(a)-\zp(a)}{\pscale(a)} \right) \wapq(x).
\end{multline}
Recalling \eqref{convequiv2} one easily sees that \eqref{E:inside2} also holds for $x<a$.
Finally, for $b\leq x \leq c$, we have using \eqref{E:inside2},
\begin{align}\label{E:over}
w(x) = &  \e_x \left[ \mathrm{e}^{-p \tau_b^-} w \left( X_{\tau_b^-} \right)  ; \tau_b^- < \tau_c^+ \right] \notag \\
= &  \e_x \left[ \mathrm{e}^{-p \tau_b^-} W^{(p+q)} \left( X_{\tau_b^-} - a \right)  ; \tau_b^- < \tau_c^+ \right]  \frac { w(b) - \zapq(b)  - \left( \frac{w(a)-\zp(a)}{\pscale(a)} \right) \wapq(b)  } {W^{(p+q)}(b-a)}  \notag \\
&+ \e_x \left[ \mathrm{e}^{-p \tau_b^-} \zapq \left( X_{\tau_b^-} \right)  ; \tau_b^- < \tau_c^+ \right] \\ & + \left( \frac{w(a)-\zp(a)}{\pscale(a)} \right) \e_x \left[ \mathrm{e}^{-p \tau_b^-} \wapq \left( X_{\tau_b^-} \right)  ; \tau_b^- < \tau_c^+ \right] .
\end{align}
By the lemmas in Section \ref{S:techlemmas} and Fubini's theorem, we have,
\begin{equation}\label{defB}
 \begin{split}
\e_x  \Big[ \mathrm{e}^{-p \tau_b^-} \wapq & \left( X_{\tau_b^-}\right) ; \tau_b^- < \tau_c^+ \Big] \\
= &  \e_x \left[ \mathrm{e}^{-p \tau_b^-} W^{(p+q)} \left( X_{\tau_b^-}\right) ; \tau_b^- < \tau_c^+ \right] \\
& - q \int_0^a \pscale(y) \e_x \left[ \mathrm{e}^{-p \tau_b^-} W^{(p+q)} \left( X_{\tau_b^-} - y\right) ; \tau_b^- < \tau_c^+ \right] \mathrm{d}y \\
= &  W^{(p+q)}(x) - q\int_b^x W^{(p)}(x-y) W^{(p+q)}(y)\mathrm{d}y \\
& - \frac{  W^{(p)}(x-b)  }{ W^{(p)}(c-b) } \left( W^{(p+q)}(c) - q\int_b^c W^{(p)}(c-y) W^{(p+q)}(y)\mathrm{d}y \right) \\
& - q \int_0^a \pscale(y) \Bigg\{ W^{(p+q)}(x-y) - q\int_b^{x} W^{(p)}(x-z) W^{(p+q)}(z-y)\mathrm{d}z \\
& - \frac{  W^{(p)}(x-b)  }{ W^{(p)}(c-b) } \left( W^{(p+q)}(c-y) - q\int_b^c W^{(p)}(c-z) W^{(p+q)}(z-y)\mathrm{d}z \right) \Bigg\} \mathrm{d}y \\
= &  \wapq(x) - q\int_b^x W^{(p)}(x-z)\wapq(z)\mathrm dz \\
& - \frac{  W^{(p)}(x-b)  }{ W^{(p)}(c-b) } \left( \wapq(c) - q\int_b^c W^{(p)}(c-z)\wapq(z)\mathrm dz \right)
 \end{split}
\end{equation}
and, similarly,
\begin{equation}\label{defA}
 \begin{split}
 \e_x  \Big[ \mathrm{e}^{-p \tau_b^-} \zapq & \left( X_{\tau_b^-}\right) ;  \tau_b^- < \tau_c^+ \Big] \\
= &  Z^{(p+q)}(x) - q \int_b^x \pscale(x-y) Z^{(p+q)}(y) \mathrm{d}y \\
& - \frac{  W^{(p)}(x-b)  }{ W^{(p)}(c-b) } \left( Z^{(p+q)}(c) - q\int_b^c W^{(p)}(c-y) Z^{(p+q)}(y)\mathrm{d}y \right) \\
& - q \int_0^a Z^{(p)}(y) \Bigg\{ W^{(p+q)}(x-y) - q\int_b^{x} W^{(p)}(x-z) W^{(p+q)}(z-y)\mathrm{d}z \\
& - \frac{  W^{(p)}(x-b)  }{ W^{(p)}(c-b) } \left( W^{(p+q)}(c-y) - q\int_b^c W^{(p)}(c-z) W^{(p+q)}(z-y)\mathrm{d}z \right) \Bigg\} \mathrm{d}y \\
= &   \zapq(x) - q\int_b^x W^{(p)}(x-z)\zapq(z)\mathrm dz \\
& - \frac{  W^{(p)}(x-b)  }{ W^{(p)}(c-b) } \left( \zapq(c) - q\int_b^c W^{(p)}(c-z)\zapq(z)\mathrm dz \right) .
 \end{split}
\end{equation}

All is left to obtain are the expressions for $w(a)$ and $w(b)$. It is here that we need the assumption that $X$ has paths of bounded variation.
Setting $x=a$ in \eqref{E:inside2}, using that  $\qscale (0) \neq 0$ because $X$ has paths of bounded variation and noticing that (cf. \eqref{convequiv2})
$$
\wapq(a) = \pscale(a) \quad \text{and} \quad \zapq(a) = \zp(a),
$$
leads to
$$
\frac{w(b)-\zapq(b)}{\wapq(b)} = \frac{w(a)-\zp(a)}{\pscale(a)} .
$$
Using the above equation once in \eqref{E:inside2} and twice in \eqref{E:over},
\begin{multline}\label{E:over3_BV}
w(x) = \e_x   \left[ \mathrm{e}^{-p \tau_b^-} \zapq \left( X_{\tau_b^-}\right) ; \tau_b^- < \tau_c^+ \right] \\
\quad + \left( \frac{w(b)-\zapq(b)}{\wapq(b)} \right) \e_x   \left[ \mathrm{e}^{-p \tau_b^-} \wapq \left( X_{\tau_b^-}\right) ; \tau_b^- < \tau_c^+ \right] ,
\end{multline}
for $x\leq c$. Setting $x=b$ in \eqref{E:over3_BV} and using \eqref{defB} and \eqref{defA} then yields
\begin{equation}\label{E:over4_BV}
\begin{split}
w(b) = &   \zapq(b) - \frac{  W^{(p)}(0)  }{ W^{(p)}(c-b) } \left( \zapq(c) - q\int_b^c W^{(p)}(c-z)\zapq(z)\mathrm dz \right)   + \left( \frac{w(b)-\zapq(b)}{\wapq(b)} \right) \\
& \times \left( \wapq(b) - \frac{  W^{(p)}(0)  }{ W^{(p)}(c-b) } \left( \wapq(c) - q\int_b^c W^{(p)}(c-z)\wapq(z)\mathrm dz \right) \right).
\end{split}
\end{equation}
Solving \eqref{E:over3_BV} for $w(b)$ leads to
$$
w(b) =  \zapq(b)   - \frac{\zapq(c) - q\int_b^c W^{(p)}(c-z)\zapq(z)\mathrm dz }{ \wapq(c) - q\int_b^c W^{(p)}(c-z)\wapq(z)\mathrm dz} \wapq(b).
$$
Plugging this into \eqref{E:over3_BV}, using \eqref{defB}-\eqref{defA}, cancelling out a few terms and rearranging, we get for all $x\leq c$,
\begin{multline*}
w(x) = \zapq(x)    - q\int_b^x W^{(p)}(x-z)\zapq(z)\mathrm dz \\  - \frac{\zapq(c) - q\int_b^c W^{(p)}(c-z)\zapq(z)\mathrm dz }{ \wapq(c) - q\int_b^c W^{(p)}(c-z)\wapq(z)\mathrm dz} \left( \wapq(x) - q\int_b^x W^{(p)}(x-z)\wapq(z)\mathrm dz \right),
\end{multline*}
which proves Theorem \ref{T:stopbelow_double} when $X$ has paths of bounded variation.

Now we assume that $X$ has paths of unbounded variation. We assume here without loss of generality that $p>0$ as the boundary case $p=0$ can be dealt with by taking limits as $p\downarrow0$. In order to prove this case, we use a similar argument as in the proof of Lemma \ref{mainlemma}. Using the notation in that proof, we have since Theorem \ref{T:stopbelow_double} has been proved for the case where the spectrally negative L\'evy process has paths of bounded variation,
\begin{multline}\label{thm1_seqn}
\e_x \left[ \mathrm{e}^{- p \tau_{0,n}^- - q \int_0^{\tau_{0,n}^-} \ind_{(a,b)} (X^n_s) \mathrm{d}s }  ; \tau_{0,n}^-<\tau_{c,n}^+ \right] = \zapqn(x)    - q\int_b^x W^{(p)}_n(x-z)\zapqn(z)\mathrm dz \\  - \frac{\zapqn(c) - q\int_b^c W^{(p)}_n(c-z)\zapqn(z)\mathrm dz }{ \wapqn(c) - q\int_b^c W^{(p)}_n(c-z)\wapqn(z)\mathrm dz} \left( \wapqn(x) - q\int_b^x W^{(p)}_n(x-z)\wapqn(z)\mathrm dz \right),
\end{multline}
where
\begin{equation*}
 \begin{split}
  \wapqn(x) :=&   W^{(p+q)}_n(x) - q \int_0^a W_n^{(p+q)}(x-y) W^{(p)}_n(y) \mathrm{d}y, \\
\zapqn(x) :=&  Z^{(p+q)}_n(x) - q \int_0^a W_n^{(p+q)}(x-y) Z^{(p)}_n(y) \mathrm{d}y.
 \end{split}
\end{equation*}
As $X^n$ converges $\mathbb P_x$-almost surely to $X$ uniformly on compact time intervals, we have, similarly to \eqref{xiaowen1}, for all $t>0$, $\mathbb P_x$-a.s.,
\begin{equation*}
\mathrm{e}^{- p \tau_{0,n}^- - q \int_0^{\tau_{0,n}^-} \ind_{(a,b)} (X^n_s) \mathrm{d}s } \mathbf 1_{\{ \tau_{0,n}^-<\tau_{c,n}^+\wedge t \}} \to \mathrm{e}^{- p \tau_{0}^- - q \int_0^{\tau_{0}^-} \ind_{(a,b)} (X_s) \mathrm{d}s } \mathbf 1_{\{ \tau_{0}^-<\tau_{c}^+\wedge t \}}.
\end{equation*}
Further, one can show similarly to \eqref{xiaowen2},
\begin{multline*}
\left| \mathrm{e}^{- p \tau_{0,n}^- - q \int_0^{\tau_{0,n}^-} \ind_{(a,b)} (X^n_s) \mathrm{d}s } \mathbf 1_{\{ \tau_{0,n}^-<\tau_{c,n}^+  \}}  - \mathrm{e}^{- p \tau_{0}^- - q \int_0^{\tau_{0}^-} \ind_{(a,b)} (X_s) \mathrm{d}s } \mathbf 1_{\{ \tau_{0}^-<\tau_{c}^+  \}} \right| \\
\leq 2\mathrm e^{-pt} + \left| \mathrm{e}^{- p \tau_{0,n}^- - q \int_0^{\tau_{0,n}^-} \ind_{(a,b)} (X^n_s) \mathrm{d}s } \mathbf 1_{\{ \tau_{0,n}^-<\tau_{c,n}^+\wedge t \}} \to \mathrm{e}^{- p \tau_{0}^- - q \int_0^{\tau_{0}^-} \ind_{(a,b)} (X_s) \mathrm{d}s } \mathbf 1_{\{ \tau_{0}^-<\tau_{c}^+\wedge t \}}  \right|,
\end{multline*}
which yields (because $p>0$),
\begin{equation*}
\mathrm{e}^{- p \tau_{0,n}^- - q \int_0^{\tau_{0,n}^-} \ind_{(a,b)} (X^n_s) \mathrm{d}s } \mathbf 1_{\{ \tau_{0,n}^-<\tau_{c,n}^+  \}}  \to \mathrm{e}^{- p \tau_{0}^- - q \int_0^{\tau_{0}^-} \ind_{(a,b)} (X_s) \mathrm{d}s } \mathbf 1_{\{ \tau_{0}^-<\tau_{c}^+  \}}.
\end{equation*}
Thus by the dominated convergence theorem it follows that the left hand side of \eqref{thm1_seqn} converges, as $n\to\infty$, to $$\e_x \left[ \mathrm{e}^{- p \tau_{0}^- - q \int_0^{\tau_{0}^-} \ind_{(a,b)} (X_s) \mathrm{d}s }  ; \tau_{0}^-<\tau_{c}^+ \right].$$
On the other hand, we have seen that $W_n^{(q)}(x)$ and $Z^{(q)}_n(x)$ converge to $W^{(q)}(x)$ and $Z^{(q)}(x)$ respectively for all $q\geq0$ and   since $W_n^{(q)}(x), Z^{(q)}_n(x)$ are increasing, positive functions it follows by the dominated convergence theorem  that $\wapqn(x)\to \wapq(x)$ and $\zapqn(x)\to\zapq(x)$ for any $x$. Since $\wapqn(x), \zapqn(x)$ are also increasing, positive functions, it then follows by the dominated convergence theorem that the right hand side of \eqref{thm1_seqn} converges to
\begin{multline*}
 \zapq(x)    - q\int_b^x W^{(p)}(x-z)\zapq(z)\mathrm dz   \\ - \frac{\zapq(c) - q\int_b^c W^{(p)}(c-z)\zapq(z)\mathrm dz }{ \wapq(c) - q\int_b^c W^{(p)}(c-z)\wapq(z)\mathrm dz} \left( \wapq(x) - q\int_b^x W^{(p)}(x-z)\wapq(z)\mathrm dz \right).
\end{multline*}
This proves Theorem \ref{T:stopbelow_double} also for the case that $X$ has paths of unbounded variation.

\section{Proof of Theorem~\ref{T:stopabove}}
The proof of this theorem is very similar to the proof of Theorem \ref{T:stopbelow_double}.
We first prove it in the case where $X$ has paths of bounded variation.
Fix $0 \leq a < b\leq c$ and $p,q \geq 0$. For $x \leq c$, define
$$
w(x) = \e_x \left[ \mathrm{e}^{-p \tau_c^+ - q \int_0^{\tau_c^+} \mathbb{I}_{(a,b)} (X_s)  \mathrm{d}s} ; \tau_c^+<\tau_0^- \right] .
$$
Using the strong Markov property of $X$, the fact that $X$ is skip-free upward and \eqref{E:exitabove}, we can write, for $0\leq x < a$,
\begin{equation}\label{E:below}
w(x) = w(a) \e_x \left[  \mathrm{e}^{-p \tau_a^+}; \tau_a^+<\tau_0^- \right] = w(a) \frac{W^{(p)}(x)}{W^{(p)}(a)}.
\end{equation}
Similarly, for $a \leq x < b$, using \eqref{E:below} and \eqref{vqiswq}, we have
\begin{align}\label{E:between}
w(x) = &  w(b) \e_x \left[ \mathrm{e}^{-(p+q) \tau_b^+} ; \tau_b^+ < \tau_a^- \right] + \e_x \left[ \mathrm{e}^{-(p+q) \tau_a^-} w \left( X_{\tau_a^-}\right) ; \tau_a^- < \tau_b^+ \right] \notag \\
= &  w(b) \frac{W^{(p+q)}(x-a)}{W^{(p+q)}(b-a)} + \frac{w(a)}{W^{(p)}(a)}  \e_x \left[ \mathrm{e}^{-(p+q) \tau_a^-} W^{(p)} \left( X_{\tau_a^-}\right) ; \tau_a^- < \tau_b^+ \right] \notag \\
= &  w(b) \frac{W^{(p+q)}(x-a)}{W^{(p+q)}(b-a)} + \frac{w(a)}{W^{(p)}(a)} \left( \wapq(x) - \frac{W^{(p+q)}(x-a)}{W^{(p+q)}(b-a)} \wapq(b) \right).
\end{align}
Note that via   \eqref{convequiv2} one can easily show that \eqref{E:between} holds also for $x<a$.
Finally, for $b \leq x \leq c$, we have
\begin{align}\label{E:above}
w(x) = &  \e_x \left[ \mathrm{e}^{-p \tau_c^+} ; \tau_c^+ < \tau_b^- \right] + \e_x \left[ \mathrm{e}^{-p \tau_b^-} w \left( X_{\tau_b^-}\right) ; \tau_b^- < \tau_c^+ \right] \notag \\
= &  \frac{\pscale(x-b)}{\pscale(c-b)} + \frac{w(b) - \frac{w(a)}{W^{(p)}(a)} \wapq(b) }{ W^{(p+q)}(b-a)} \e_x \left[ \mathrm{e}^{-p \tau_b^-} W^{(p+q)} \left( X_{\tau_b^-} - a \right) ; \tau_b^- < \tau_c^+ \right] \notag \\
& \qquad  + \frac{w(a)}{W^{(p)}(a)} \e_x \left[ \mathrm{e}^{-p \tau_b^-} \wapq \left( X_{\tau_b^-}\right) ; \tau_b^- < \tau_c^+ \right]  .
\end{align}


We need to obtain the expressions for $w(a)$ and $w(b)$. As we assumed that  $X$ has paths of bounded variation, we have $\qscale (0) \neq 0$ and thus setting $x=a$ in \eqref{E:between} 
yields
\begin{equation*}
 \frac{ w(b)}{\wapq(b)} = \frac{ w(a)}{W^{(p)}(a)}.
\end{equation*}

Plugging this  into \eqref{E:between} and \eqref{E:above} using \eqref{defB} yields,
\begin{equation} \label{E:between3_BV}
 \begin{split}
 w(x) = &  w(b)\frac{ \wapq(x)}{\wapq(b)}, \quad x<b, \\
w(x) = &  \frac{\pscale(x-b)}{\pscale(c-b)}  + \frac{ w(b)}{\wapq(b)}  \Bigg\{ \wapq(x)  - q\int_b^x W^{(p)}(x-z)\wapq(z)\mathrm dz \\
& - \frac{  W^{(p)}(x-b)  }{ W^{(p)}(c-b) } \left( \wapq(c) - q\int_b^c W^{(p)}(c-z)\wapq(z)\mathrm dz \right) \Bigg\}, \quad b\leq x\leq c.
 \end{split}
\end{equation}

Setting $x=b$ in \eqref{E:between3_BV} gives us
\begin{equation*}
w(b) = \frac{  \wapq(b) }{ \wapq(c) - q\int_b^c W^{(p)}(c-z)\wapq(z)\mathrm dz }
\end{equation*}
and plugging this into \eqref{E:between3_BV} leads to
\begin{equation*}
 w(x) =\frac{  \wapq(x)  - q\int_b^x W^{(p)}(x-z)\wapq(z)\mathrm dz }{ \wapq(c) - q\int_b^c W^{(p)}(c-z)\wapq(z)\mathrm dz },
\end{equation*}
for all $x\leq c$.
This proves Theorem \ref{T:stopabove} when $X$ has paths of bounded variation. The case where $X$ has paths of unbounded variation follows using the same arguments as in the proof of Theorem \ref{T:stopbelow_double}.


%


\section{Proof of corollaries}

We will prove the corollaries only for $p>0$ and $q>0$. The cases where $p=0$ or $q=0$, then follow  by taking limits as $p\downarrow0$ or $q\downarrow0$. For the proofs we will make heavy use of the fact that (cf. \cite{kyprianou2006}*{Lemma 8.4}) the scale function can be written for $q,x\geq0$ as
\begin{equation}\label{scaledecomp}
 W^{(q)}(x)= \mathrm e^{\Phi(q)x} W_{\Phi(q)}(x),
\end{equation}
where $W_{\Phi(q)}(x)$ is the 0-scale function of the spectrally negative L\'evy process with Laplace exponent $\psi_{\Phi(q)}(\theta):=\psi(\Phi(q)+\theta)-q$. Further (cf. \cite{kyprianou2006}*{8.7}), $$W_{\Phi(q)}(\infty):=\lim_{x\to\infty}W_{\Phi(q)}(x)=\frac{1}{\psi_{\Phi(q)}'(0+)} =\frac{1}{\psi'(\Phi(q))}, $$ which implies that $W_{\Phi(q)}(\infty)<\infty$ except if simultaneously $q=0$ and $\psi'(0+)=0$.

\subsection{Proof of Corollary \ref{corol1}}
(i). Taking limits as $c\to\infty$ in Theorem \ref{T:stopbelow_double}, we see  that we need to show
\begin{equation}\label{helpcorol1}
\lim_{c\to\infty} \frac{\zapq(c) - q\int_b^c W^{(p)}(c-z)\zapq(z)\mathrm dz }{ \wapq(c) - q\int_b^c W^{(p)}(c-z)\wapq(z)\mathrm dz}
=  \frac{ \frac{p}{\Phi(p)} + q \int_a^b \mathrm e^{-\Phi(p)y} \zapq(y)\mathrm dy }{ 1  + q \int_a^b \mathrm e^{-\Phi(p)y} \wapq(y)\mathrm dy }.
\end{equation}
Using \eqref{convequiv1} and \eqref{scaledecomp}, it follows by the dominated convergence theorem (recalling that we assumed without loss of generality $p>0$),
\begin{equation*}
 \lim_{c\to\infty} \frac{ \wapq(c) - q\int_b^c W^{(p)}(c-z)\wapq(z)\mathrm dz}{W^{(p)}(c)}
=  1  + q \int_a^b \mathrm e^{-\Phi(p)y} \wapq(y)\mathrm dy
\end{equation*}
and similarly, using also \cite{kyprianou2006}*{Exercise 8.5(i)},
\begin{equation*}
 \lim_{c\to\infty} \frac{ \zapq(c) - q\int_b^c W^{(p)}(c-z)\zapq(z)\mathrm dz }{W^{(p)}(c)}
=   \frac{p}{\Phi(p)} + q \int_a^b \mathrm e^{-\Phi(p)y} \zapq(y)\mathrm dy.
\end{equation*}
Now \eqref{helpcorol1} follows.

\bigskip

(ii). Setting $b=c$ and taking limits as $c\to\infty$ in Theorem \ref{T:stopbelow_double}, we see that we need to show
\begin{equation}\label{helpcorol1ii}
\lim_{c\to\infty} \frac{\zapq(c) }{ \wapq(c)}
=  \frac{ \frac{p+q}{\Phi(p+q)} - q\int_0^a \mathrm e^{-\Phi(p+q)y} Z^{(p)}(y)\mathrm dy }{ 1 - q\int_0^a \mathrm e^{-\Phi(p+q)y} W^{(p)}(y)\mathrm dy  }.
\end{equation}
Recalling \eqref{convequiv2} and \eqref{scaledecomp}, it follows by the dominated convergence theorem (recalling that we assumed without loss of generality $p,q>0$),
\begin{equation*}
 \lim_{c\to\infty} \frac{ \wapq(c)}{W^{(p+q)}(c)}
= 1 - q\int_0^a \mathrm e^{-\Phi(p+q)y} W^{(p)}(y)\mathrm dy
\end{equation*}
and similarly, using again \cite{kyprianou2006}*{Exercise 8.5(i)},
\begin{equation*}
 \lim_{c\to\infty} \frac{ \zapq(c)}{W^{(p+q)}(c)}
=   \frac{p+q}{\Phi(p+q)} - q\int_0^a \mathrm e^{-\Phi(p+q)y} Z^{(p)}(y)\mathrm dy.
\end{equation*}
Now \eqref{helpcorol1ii} follows.

\subsection{Proof of Corollary \ref{corol2}}
(i). Using spatial homogeneity and Theorem \ref{T:stopabove} for sufficiently large $m$,
\begin{equation*}
 \begin{split}
  \e_x \Big[ \mathrm{e}^{- p \tau_c^+ - q \int_0^{\tau_c^+} \ind_{(a,b)} (X_s) \mathrm{d}s } & ;  \tau_c^+<\infty \Big] \\
= & \lim_{m\to\infty} \e_x \left[ \mathrm{e}^{- p \tau_c^+ - q \int_0^{\tau_c^+} \ind_{(a,b)} (X_s) \mathrm{d}s } ;  \tau_c^+<\tau_{-m}^- \right] \\
= & \lim_{m\to\infty}  \e_{x+m} \left[ \mathrm{e}^{- p \tau_{c+m}^+ - q \int_0^{\tau_{c+m}^+} \ind_{(a+m,b+m)} (X_s) \mathrm{d}s } ;  \tau_{c+m}^+<\tau_{0}^- \right] \\
= & \lim_{m\to\infty} \frac{  \mathcal W_{a+m}^{(p,q)}(x+m)  - q\int_{b+m}^{x+m} W^{(p)}(x+m-z)\mathcal W_{a+m}^{(p,q)}(z) \mathrm dz }{ \mathcal W_{a+m}^{(p,q)}(c+m) - q\int_{b+m}^{c+m} W^{(p)}(c+m-z)\mathcal W_{a+m}^{(p,q)}(z)\mathrm dz } \\
= & \lim_{m\to\infty} \frac{  \mathcal W_{a+m}^{(p,q)}(x+m)  - q\int_{b}^{x} W^{(p)}(x-y)\mathcal W_{a+m}^{(p,q)}(y+m) \mathrm dy }{ \mathcal W_{a+m}^{(p,q)}(c+m) - q\int_{b}^{c} W^{(p)}(c-y)\mathcal W_{a+m}^{(p,q)}(y+m)\mathrm dy } \\
= &  \frac{   \mathcal H^{(p,q)}(x-a) - q\int_{b}^{x} W^{(p)}(x-y) \mathcal H^{(p,q)}(y-a) \mathrm dy }{  \mathcal H^{(p,q)}(c-a) - q\int_{b}^{c} W^{(p)}(c-y) \mathcal H^{(p,q)}(y-a)\mathrm dy },
 \end{split}
\end{equation*}
where the last line follows by the dominated convergence theorem (noting that $\mathcal H^{(p,q)}$ is an increasing function) and
\begin{equation*}
\begin{split}
 \lim_{m\to\infty} \frac{ \mathcal W_{a+m}^{(p,q)}(x+m) }{ W^{(p)}(m) } = &
\lim_{m\to\infty} \frac{W^{(p)}(x+m) + q \int_a^x W^{(p+q)}(x-y)W^{(p)}(y+m)\mathrm d y  }{ W^{(p)}(m) } \\
= &  \mathrm e^{\Phi(p)x} +  q\int_a^x  W^{(p+q)}(x-y) \mathrm e^{\Phi(p)y} \mathrm d y \\
= &  \mathrm e^{\Phi(p)a} \mathcal H^{(p,q)}(x-a),
\end{split}
\end{equation*}
which follows by again  the dominated convergence theorem and \eqref{scaledecomp}.

\bigskip

(ii). Using spatial homogeneity and Theorem \ref{T:stopabove},
\begin{equation*}
 \begin{split}
  \e_x \Big[ \mathrm{e}^{- p \tau_c^+ - q \int_0^{\tau_c^+} \ind_{(-\infty,b)} (X_s) \mathrm{d}s } & ;  \tau_c^+<\infty \Big] \\
= & \lim_{m\to\infty} \e_x \left[ \mathrm{e}^{- p \tau_c^+ - q \int_0^{\tau_c^+} \ind_{(-m,b)} (X_s) \mathrm{d}s } ;  \tau_c^+<\tau_{-m}^- \right] \\
= & \lim_{m\to\infty}  \e_{x+m} \left[ \mathrm{e}^{- p \tau_{c+m}^+ - q \int_0^{\tau_{c+m}^+} \ind_{(0,b+m)} (X_s) \mathrm{d}s } ;  \tau_{c+m}^+<\tau_{0}^- \right] \\
= & \lim_{m\to\infty} \frac{ W^{(p+q)}(x+m)  - q\int_{b+m}^{x+m} W^{(p)}(x+m-z) W^{(p+q)}(z) \mathrm dz }{ W^{(p+q)}(c+m) - q\int_{b+m}^{c+m} W^{(p)}(c+m-z) W^{(p+q)}(z)\mathrm dz } \\
= & \lim_{m\to\infty} \frac{  W^{(p+q)}(x+m)  - q\int_{b}^{x} W^{(p)}(x-y)W^{(p+q)}(y+m) \mathrm dy }{  W^{(p+q)}(c+m) - q\int_{b}^{c} W^{(p)}(c-y)W^{(p+q)}(y+m)\mathrm dy } \\
= & \frac{  \mathrm e^{\Phi(p+q)x} \left( 1 - q\int_{0}^{x-b} \mathrm e^{-\Phi(p+q)y}W^{(p)}(y)   \mathrm dy \right) }{    \mathrm e^{\Phi(p+q)c} \left( 1 - q\int_{0}^{c-b} \mathrm e^{-\Phi(p+q)y}W^{(p)}(y)   \mathrm dy \right) },
 \end{split}
\end{equation*}
where the last equality is due to
\begin{multline*}
 \lim_{m\to\infty} \frac{ W^{(p+q)}(x+m)  - q\int_{b}^{x} W^{(p)}(x-y)W^{(p+q)}(y+m) \mathrm dy  }{ W^{(p+q)}(m) } \\
=  \mathrm e^{\Phi(p+q)x} -  q\int_b^x  W^{(p)}(x-y) \mathrm e^{\Phi(p+q)y} \mathrm d y,
\end{multline*}
which follows from  the dominated convergence theorem and \eqref{scaledecomp}.

\subsection{Proof of Corollary \ref{corol3}}
(i). Assume $\psi'(0+)>0$, which implies $\Phi(0)=0$.  Then since $\tau_c^+<\infty$ almost surely, we get using Corollary \ref{corol2}(i) with $p=0$, noting that $\mathcal H^{(0,q)}(x)=Z^{(q)}(x)$,
\begin{equation*}
\begin{split}
\e_x \left[ \mathrm{e}^{  - q \int_0^{\infty} \ind_{(a,b)} (X_s) \mathrm{d}s }  \right]
 = &  \lim_{c\to\infty} \e_x \left[ \mathrm{e}^{  - q \int_0^{\tau_c^+} \ind_{(a,b)} (X_s) \mathrm{d}s }  ; \tau_c^+<\infty \right] \\
= & \lim_{c\to\infty} \frac{ Z^{(q)}(x-a) - q\int_{b}^{x} W(x-y) Z^{(q)}(y-a) \mathrm dy }{Z^{(q)}(c-a) - q\int_{b}^{c} W(c-y) Z^{(q)}(y-a) \mathrm dy }.
\end{split}
\end{equation*}
Using \eqref{identH}, we deduce using the dominated convergence theorem,
\begin{equation*}
 \begin{split}
  \lim_{c\to\infty} \left(  Z^{(q)}(c-a) - q\int_{b}^{c} W(c-y) Z^{(q)}(y-a) \mathrm dy \right)
=  & \lim_{c\to\infty} \left( 1 + q\int_a^b W(c-y)  Z^{(q)}(y-a) \mathrm dy \right) \\
= &  1 + qW(\infty)\int_a^b   Z^{(q)}(y-a) \mathrm dy \\
= &  1+ \frac{q}{\psi'(0+)} \int_0^{b-a}   Z^{(q)}(y) \mathrm dy,
 \end{split}
\end{equation*}
which proves Corollary \ref{corol3}(i).

\bigskip

(ii). Similarly as for part (i), we have now using  Corollary \ref{corol2}(ii) with $p=0$,
\begin{equation*}
\begin{split}
\e_x \left[ \mathrm{e}^{  - q \int_0^{\infty} \ind_{(-\infty,b)} (X_s) \mathrm{d}s }  \right]
 = &  \lim_{c\to\infty} \e_x \left[ \mathrm{e}^{  - q \int_0^{\tau_c^+} \ind_{(-\infty,b)} (X_s) \mathrm{d}s }  ; \tau_c^+<\infty \right] \\
= & \lim_{c\to\infty} \frac{ \mathcal H^{(q,-q)}(x-b) }{ \mathcal H^{(q,-q)}(c-b) }.
\end{split}
\end{equation*}
Further, using \eqref{def_scale} and l'H\^opital's rule,
\begin{equation*}
 \begin{split}
\lim_{c\to\infty}  \mathcal H^{(q,-q)}(c-b)
 = & \lim_{c\to\infty}  \mathrm e^{\Phi(q)(c-b)} \left( 1 - q\int_{0}^{c-b} \mathrm e^{-\Phi(q)y}W(y)   \mathrm dy \right) \\
= &  \lim_{c\to\infty} \frac{q\int_{c-b}^\infty \mathrm e^{-\Phi(q)y}W(y)   \mathrm dy  }{\mathrm e^{-\Phi(q)(c-b)} } \\
= &    \frac{qW(\infty) }{\Phi(q)} \\
= &  \frac {q }{\psi'(0+)\Phi(q)},
 \end{split}
\end{equation*}
which proves Corollary \ref{corol3}(ii).

\bigskip

(iii). Assume $\psi'(0+)<0$, which implies $\Phi(0)>0$. Then since $\tau_{-m}^-<\infty$ almost surely for any $m>0$, we get using spatial homogeneity and Corollary \ref{corol1}(i) for $p=0$ and sufficiently large $m$,
\begin{equation*}
 \begin{split}
  \e_x \left[ \mathrm{e}^{  - q \int_0^{\infty} \ind_{(a,b)} (X_s) \mathrm{d}s }  \right]
= & \lim_{m\to\infty} \e_x \left[ \mathrm{e}^{ - q \int_0^{\tau_{-m}^-} \ind_{(a,b)} (X_s) \mathrm{d}s } ;  \tau_{-m}^-<\infty \right] \\
= & \lim_{m\to\infty}  \e_{x+m} \left[  \mathrm{e}^{ - q \int_0^{\tau_0^-} \ind_{(a+m,b+m)} (X_s) \mathrm{d}s } ;  \tau_{0}^-<\infty \right] \\
= &  Z^{(q)}(x-a)  - q\int_{b}^{x} W^{}(x-y)Z^{(q)}(y-a) \mathrm dy \\
& -  \lim_{m\to\infty} \frac{ \mathrm e^{-\Phi(0)m} q \mathrm e^{-\Phi(0)a}\int_{0}^{b-a} \mathrm e^{-\Phi(0)y} Z^{(q)}(y) \mathrm dy }{ 1  + \mathrm e^{-\Phi(0)m} q \int_{a}^{b} \mathrm e^{-\Phi(0)z} \mathcal W_{a+m}^{(0,q)}(z+m)\mathrm dz } \\
& \times \left( \mathcal W_{a+m}^{(0,q)}(x+m) - q\int_{b}^{x} W^{}(x-z) \mathcal W_{a+m}^{(0,q)}(z+m)\mathrm dz \right).
 \end{split}
\end{equation*}
Note that in the above we used that $\mathcal Z_a^{(0,q)}(x)=Z^{(q)}(x-a)$.
 Now we have by the dominated convergence theorem and \eqref{scaledecomp},
\begin{equation}\label{ratiowapq}
\begin{split}
 \lim_{m\to\infty} \frac{ \mathcal W_{a+m}^{(0,q)}(x+m) }{W(m)} = & \lim_{m\to\infty} \frac{ W(x+m) + q \int_a^x W^{(q)}(x-y)W(y+m)\mathrm d y }{W(m)} \\
= &  \mathrm e^{\Phi(0)x}  + q \int_a^x W^{(q)}(x-y)\mathrm e^{\Phi(0)y} \mathrm d y \\
= &   \mathrm e^{\Phi(0)a}\mathcal H^{(0,q)}(x-a)
\end{split}
\end{equation}
and thus also,
\begin{equation*}
\begin{split}
 \lim_{m\to\infty} \frac{ \mathrm e^{\Phi(0)m}   + q \int_{a}^{b} \mathrm e^{-\Phi(0)z} \mathcal W_{a+m}^{(0,q)}(z+m)\mathrm dz  }{W(m)}
= &   \frac1{W_{\Phi(0)}(\infty)}   + q \int_{a}^{b} \mathrm e^{-\Phi(0)z} \mathrm e^{\Phi(0)a}\mathcal H^{(0,q)}(z-a)\mathrm dz  \\
= &   \psi'(\Phi(0))   + q \int_{0}^{b-a} \mathrm e^{-\Phi(0)y} \mathcal H^{(0,q)}(y)\mathrm dy.
\end{split}
\end{equation*}
 Combining all three computations gives us Corollary \ref{corol3}(iii).

\bigskip

(iv). Similarly,  as for part (iii), we have now using  Corollary \ref{corol1}(ii) with $p=0$ and noting that $\mathcal Z_a^{(0,q)}(x)=Z^{(q)}(x-a)$,
\begin{equation*}
 \begin{split}
  \e_x \left[ \mathrm{e}^{  - q \int_0^{\infty} \ind_{(a,\infty)} (X_s) \mathrm{d}s }  \right]
= & \lim_{m\to\infty} \e_x \left[ \mathrm{e}^{ - q \int_0^{\tau_{-m}^-} \ind_{(a,\infty)} (X_s) \mathrm{d}s } ;  \tau_{-m}^-<\infty \right] \\
= & \lim_{m\to\infty}  \e_{x+m} \left[  \mathrm{e}^{ - q \int_0^{\tau_0^-} \ind_{(a+m,\infty)} (X_s) \mathrm{d}s } ;  \tau_{0}^-<\infty \right] \\
= &  Z^{(q)}(x-a)   -  \lim_{m\to\infty} \frac{ \frac{q}{\Phi(q)} - q\int_0^{a+m} \mathrm e^{-\Phi(q)y} \mathrm dy }{ 1 - q\int_0^{a+m} \mathrm e^{-\Phi(q)y} W(y)\mathrm dy  } \mathcal W_{a+m}^{(0,q)}(x+m) \\
= &  Z^{(q)}(x-a)   -  \mathrm e^{\Phi(0)a} \mathcal H^{(0,q)}(x-a)\lim_{m\to\infty} \frac{  W(m)  \frac1{\Phi(q)} \mathrm e^{-\Phi(q)(a+m)}  }{  \int_{a+m}^\infty \mathrm e^{-\Phi(q)y} W(y)\mathrm dy  },
 \end{split}
\end{equation*}
where in the last line we used \eqref{def_scale} and \eqref{ratiowapq}.
By \eqref{scaledecomp} and  l'H\^opital's rule,
\begin{equation*}
 \begin{split}
\lim_{m\to\infty} \frac{  W(m)  \frac1{\Phi(q)} \mathrm e^{-\Phi(q)(a+m)}  }{  \int_{a+m}^\infty \mathrm e^{-\Phi(q)y} W(y)\mathrm dy  } = &  \frac{ W_{\Phi(0)}(\infty)\mathrm e^{-\Phi(q)a} }{  \Phi(q)  } \lim_{m\to\infty} \frac{    \mathrm e^{-(\Phi(q)-\Phi(0))m}  }{  \int_{a+m}^\infty \mathrm e^{-(\Phi(q)-\Phi(0))y} W_{\Phi(0)}(y)\mathrm dy  } \\
= &   \frac{\Phi(q)-\Phi(0) }{ \Phi(q) }  \mathrm e^{-\Phi(0)a}
 \end{split}
\end{equation*}
and in combination with the previous computation, this proves Corollary \ref{corol3}(iv).

\section{Applications}

\subsection{Perpetual double knock-out corridor options in an exponential spectrally negative L\'evy model}

We assume that the price process of an underlying security is given by $(\mathrm e^{X_t})_{t\geq0}$ under the risk-neutral measure $\mathbb P$. For this model (which includes the Black-Scholes model) we would like to price a so-called (European) perpetual double knock-out corridor  option. In a corridor option (see e.g. Pechtl \cite{pechtl}), the payoff function is the amount of time the underlying spends in a given interval, the so-called corridor. For our  option we include the feature, similar to barrier options, that the option expires when the price process leaves a predetermined interval. In particular, if we assume that the corridor is given by $(\mathrm e^a, \mathrm e^b)$ and the option gets knocked out when the price process leaves the interval $[\mathrm e^0, \mathrm e^c]$ with $0\leq a< b\leq c<\infty$, then the price of the option equals
\begin{equation*}
V(x) := \mathbb E_x \left[\mathrm e^{-p(\tau_0^-\wedge\tau_c^+)} \int_0^{\tau_0^-\wedge\tau_c^+} \ind_{(a,b)}(X_s)\mathrm ds  \right],
\end{equation*}
where $p\geq0$ is the risk-free interest rate and $\mathrm e^x$ is the initial price of the security.

From Theorem \ref{T:stopbelow_double} and Theorem \ref{T:stopabove} in combination with the dominated convergence theorem (which justifies switching derivative and expectation), we have for $x\in[0,c]$,

\begin{equation*}
 \begin{split}
  V(x) \\
= &   \mathbb E_x \left[\mathrm e^{-p\tau_0^-} \int_0^{\tau_0^-} \ind_{(a,b)}(X_s)\mathrm ds  ;  \tau_0^-<\tau_c^+ \right] +  \mathbb E_x \left[\mathrm e^{-p\tau_c^+} \int_0^{\tau_c^+} \ind_{(a,b)}(X_s)\mathrm ds  ;  \tau_c^+<\tau_0^- \right] \\
= &  \left.\frac{-\mathrm d }{ \mathrm dq}\right|_{q=0}  \left( \e_x \left[ \mathrm{e}^{- p \tau_0^- - q \int_0^{\tau_0^-} \ind_{(a,b)} (X_s) \mathrm{d}s } ; \tau_0^- < \tau_c^+ \right] + \e_x \left[ \mathrm{e}^{- p \tau_c^+ - q \int_0^{\tau_c^+} \ind_{(a,b)} (X_s) \mathrm{d}s } ; \tau_c^+ < \tau_0^- \right] \right) \\
= &  \left.\frac{-\mathrm d }{ \mathrm dq}\right|_{q=0} \Bigg( \zapq(x)    - q\int_b^x W^{(p)}(x-z)\zapq(z)\mathrm dz \\
& - \frac{ \wapq(x) - q\int_b^x W^{(p)}(x-z)\wapq(z)\mathrm dz }{ \wapq(c) - q\int_b^c W^{(p)}(c-z)\wapq(z)\mathrm dz} \\
& \times \left(  \zapq(c) - q\int_b^c W^{(p)}(c-z)\zapq(z)\mathrm dz  - 1 \right) \Bigg).
 \end{split}
\end{equation*}
From \eqref{fullconvolutions}, we deduce with help of the dominated convergence theorem
\begin{equation*}
\frac{\mathrm d }{ \mathrm dq} \left( W^{(p+q)}(x)   \right)_{q=0} = \lim_{q\downarrow0} \frac{W^{(p+q)}(x)-W^{(p)}(x) }{q} = \int_0^x W^{(p)}(x-y)W^{(p)}(y)\mathrm dy
\end{equation*}
and thus using \eqref{convequiv2} and again the dominated convergence theorem and noting that $W^{(p)}(z)=0$ for $z<0$,
\begin{equation*}
\begin{split}
\frac{\mathrm d }{ \mathrm dq} \Big( \wapq(x) - q\int_b^x W^{(p)}(x-z)  & \wapq(z)\mathrm dz  \Big)_{q=0} \\
= & \int_a^x W^{(p)}(x-y)W^{(p)}(y)\mathrm dy - \int_b^x W^{(p)}(x-y)\mathcal W_a^{(p,0)}(y)\mathrm dy \\
= & \int_a^{b} W^{(p)}(x-y)W^{(p)}(y)\mathrm dy
\end{split}
\end{equation*}
and
\begin{equation*}
\begin{split}
\frac{\mathrm d }{ \mathrm dq} \Big( \zapq(x)  - q\int_b^x W^{(p)}(x-z) & \zapq(z)\mathrm dz  \Big)_{q=0} \\
= & \int_a^x W^{(p)}(x-y)Z^{(p)}(y)\mathrm dy - \int_b^x W^{(p)}(x-y)\mathcal Z_a^{(p,0)}(y)\mathrm dy \\
= & \int_a^{b} W^{(p)}(x-y)Z^{(p)}(y)\mathrm dy.
\end{split}
\end{equation*}
Hence using all of the above, we get in the end
\begin{equation}\label{priceoption}
 \begin{split}
V(x)
= &  \int_a^{b}  \left(  Z^{(p)}(y) - \frac{Z^{(p)}(c) -1  }{ W^{(p)}(c) }  W^{(p)}(y)   \right) \left( \frac{ W^{(p)}(x)}{ W^{(p)}(c) }  W^{(p)}(c-y) - W^{(p)}(x-y) \right) \mathrm dy .
 \end{split}
\end{equation}
The identity \eqref{priceoption} can also be derived as follows. Using Fubini's theorem and the Markov property, we get
\begin{equation*}
 \begin{split}
V(x) = & \int_0^\infty \mathrm e^{-pt} \int_0^t \mathbb P_x \left( \tau_0^-\wedge\tau_c^+\in\mathrm dt, X_s\in(a,b) \right) \mathrm ds \\
= & \int_0^\infty \left( \int_s^\infty \mathrm e^{-pt}  \mathbb P_x(\tau_0^-\wedge\tau_c^+\in\mathrm dt, X_s\in(a,b)) \right) \mathrm ds \\
= &  \int_0^\infty \left( \mathbb E_x\left[ \int_0^\infty \mathrm e^{-pt}   \mathbb P_x \left(\tau_0^-\wedge\tau_c^+\in\mathrm dt, s<\tau_0^-\wedge\tau_c^+, X_s\in(a,b) |\mathcal F_s \right) \right] \right) \mathrm ds \\
= &  \int_0^\infty \left( \mathbb E_x \left[ \mathrm e^{-ps} \mathbf 1_{\{s<\tau_0^-\wedge\tau_c^+, X_s\in(a,b)\}}  \int_0^\infty \mathrm e^{-pt}   \mathbb P_{X_s}(\tau_0^-\wedge\tau_c^+\in\mathrm dt) \right] \right) \mathrm ds \\
= & \int_a^b \mathbb E_y \left[ \mathrm e^{-p(\tau_0^-\wedge\tau_c^+)} \right]  \int_0^\infty \mathrm e^{-ps} \mathbb P_x(s<\tau_0^-\wedge\tau_c^+, X_s\in \mathrm dy)\mathrm ds.
 \end{split}
\end{equation*}
Now \eqref{priceoption} follows using \eqref{E:exitabove} and \eqref{E:exitbelow} together with the following known formula for
 the potential measure of $X$ killed on exiting $[0,c]$:
\begin{equation*}
 \int_0^\infty \mathrm e^{-ps} \mathbb P_x(s<\tau_0^-\wedge\tau_c^+, X_s\in \mathrm dy)\mathrm ds =  \left( \frac{ W^{(p)}(x)}{ W^{(p)}(c) }  W^{(p)}(c-y) - W^{(p)}(x-y) \right) \mathrm dy,
\end{equation*}
 c.f. \cite{kyprianou2006}*{Theorem 8.7}.

Using the above methods, we can of course also price corridor options with a single knock-out feature or with the corridor being an interval of infinite length.

\subsection{Probability of bankruptcy for an Omega L\'evy risk process}

Our results can also be applied to the so-called Omega model (for some specific rate functions) introduced in \cite{AGS11} and further investigated in \cite{GSY12}. Intuitively in such a model bankruptcy (instead of ruin) occurs at rate $\omega(x)$ when the surplus process $X=(X_s)_{s\geq 0}$  is at level $x$. To be more precise, given the rate function $\omega:\mathbb R\to[0,\infty)$  the bankruptcy time $T_\omega$ can be defined as
\begin{equation*}
 T_\omega=\inf\{t>0: \int_0^t \omega (X_s)\mathrm ds> \mathbf e_1 \},
\end{equation*}
where $\mathbf e_1$ is an independent exponentially distributed random variable with parameter 1.
Typically, the rate function $\omega$ is chosen to be a decreasing function equalling zero on the positive half line so that bankruptcy does not occur when the surplus is positive.

In order to connect with the results in Section \ref{sec_mainresults}, we choose for some $b,q>0$ the bankruptcy rate as
\[
\omega(x)= \begin{cases}
0 & \text{if $x\geq0$}, \\
            q & \text{if $-b\leq x<0$}, \\
\infty & \text{if $x<-b$}.
           \end{cases}
\]
Then bankruptcy occurs at rate $q$ when $X$ is between $-b$ and $0$ and bankruptcy occurs immediately when $X$ is below level $-b$.
Suppose that the positive loading condition holds, i.e. $\e[X_1]=\psi'(0+)>0$; this implies that bankruptcy does not happen almost surely. Then for any $x\in\mathbb R$, the probability that bankruptcy never occurs is
\begin{equation*}
\begin{split}
\mathbb{P}_x (T_\omega=\infty) = & \mathbb P_x \left( \int_0^\infty \omega (X_s)\mathrm ds \leq \mathbf e_1 \right) \\
= & \mathbb E_x \left[ \mathrm e^{- \int_0^\infty \omega (X_s)\mathrm ds } \right] \\
= & \mathbb{E}_x\left[\mathrm{e}^{-q\int_0^\infty\ind_{(-b,0)}(X_s)\mathrm ds };\tau^-_{-b}=\infty \right].
\end{split}
\end{equation*}
Hence by  spatial homogeneity,  Theorem~\ref{T:stopabove},  \eqref{convequiv1} and \eqref{scaledecomp} in combination with the dominated convergence theorem,
\begin{equation*}
\begin{split}
\mathbb{P}_x (T_\omega=\infty) = & \lim_{c\rightarrow\infty}\mathbb{E}_x\left[\mathrm{e}^{-q\int_0^{\tau^+_c} \ind_{(-b,0)}(X_s) \mathrm{d}s}; \tau^+_c<\tau^-_{-b} \right]\\
= & \lim_{c\rightarrow\infty}\mathbb{E}_{x+b}\left[\mathrm{e}^{-q\int_0^{\tau^+_{c+b}} \ind_{(0,b)}(X_s) \mathrm{d}s}; \tau^+_{c+b}<\tau^-_0\right]\\
= & \lim_{c\rightarrow\infty}
\frac{W^{(q)}(x+b)-q\int_b^{b+x}W(x+b-z)W^{(q)}(z) \mathrm{d}z }{ W^{(q)}(c+b)-q\int_b^{c+b} W(c+b-z)W^{(q)}(z) \mathrm{d}z}\\
= & \lim_{c\rightarrow\infty} \frac{W(x+b)+q\int_0^ bW(x+b-z)W^{(q)}(z) \mathrm{d}z }{ W(c+b)+q\int_0^bW(c+b-z)W^{(q)}(z) \mathrm{d}z}\\
= & \frac{W(x+b)+q\int_0^bW(x+b-z)W^{(q)}(z) \mathrm{d}z}{\frac{1}{\psi'(0+)}+\frac{q}{\psi'(0+)}\int_0^bW^{(q)}(z) \mathrm{d}z}\\
= & \psi'(0+) \frac{W(x+b)+q\int_0^bW(x+b-z)W^{(q)}(z) \mathrm{d}z }{ Z^{(q)}(b)}.
\end{split}
\end{equation*}
Similarly,  the probability that bankruptcy occurs due to the surplus process dropping below the level $-b$ is given by
\begin{equation*}
\begin{split}
\mathbb{P}_x(X_{T_\omega}< -b, T_\omega<\infty)
= & \mathbb P_x \left( \int_0^{\tau_{-b}^-} \omega (X_s)\mathrm ds \leq \mathbf e_1, \tau_{-b}^-<\infty \right) \\
= &  \mathbb{E}_x\left[\mathrm{e}^{-q\int_0^{\tau^-_{-b}} \ind_{(-b,0)}(X_s) \mathrm{d}s};\tau^-_{-b}<\infty\right]\\
= &   \mathbb{E}_{x+b}\left[\mathrm{e}^{-q\int_0^{\tau^-_0} \ind_{(0,b)}(X_s) \mathrm{d}s};\tau^-_0<\infty\right],
\end{split}
\end{equation*}
which, by Corollary \ref{corol1} and \eqref{convequiv1}, equals
\begin{equation*}
\begin{split}
\mathbb{P}_x(X_{T_\omega}< -b, & T_\omega<\infty) \\
 = & Z^{(q)}(x+b)-q\int_b^{x+b}W(x+b-z)Z^{(q)}(z) \mathrm{d}z \\
& - \frac{\psi'(0+)+q\int_0^b Z^{(q)}(y) \mathrm{d}y}{1+q\int_0^b W^{(q)}(y) \mathrm{d}y} \left(W^{(q)}(x+b)-q\int_b^{b+x}W(x+b-z)W^{(q)}(z) \mathrm{d}z\right) \\
= & 1+q\int_0^{b}W(x+b-z)Z^{(q)}(z) \mathrm{d}z \\
& -\frac{\psi'(0+) + q\int_0^b Z^{(q)}(y) \mathrm{d}y }{ Z^{(q)}(b)}
\left(W(x+b)+q\int_0^{b}W(x+b-z)W^{(q)}(z) \mathrm{d}z\right).
\end{split}
\end{equation*}
In addition, the probability that bankruptcy occurs while the surplus is between $-b$ and $0$ is
\begin{equation*}
\begin{split}
\mathbb{P}_x(-b \leq X_{T_\omega}<0, T_\omega<\infty )
= & 1-\mathbb{P}_x (T_\omega=\infty)-\mathbb{P}_x(X_{T_\omega}< - b, T_\omega<\infty)\\
= &  -q\int_0^{b}W(x+b-z)Z^{(q)}(z) \mathrm{d}z\\
&\quad+\frac{q\int_0^b Z^{(q)}(y) \mathrm{d}y}{Z^{(q)}(b)}\left(W(x+b)+q\int_0^{b}W(x+b-z)W^{(q)}(z) \mathrm{d}z\right).
\end{split}
\end{equation*}


%
%
%
\section{Acknowledgements}

Funding in support of this work was provided by the Natural Sciences and Engineering Research Council of Canada (NSERC). J.-F. Renaud also acknowledges financial support from the Fonds de recherche du Qu\'ebec - Nature et technologies (FRQNT) and the Insti\-tut de finance math\'ematique de Montr\'eal (IFM2).

%
%

%
%
\end{document}